\documentclass[twoside]{article}

\usepackage[accepted]{aistats2020}
\usepackage[utf8]{inputenc} 
\usepackage{hyperref}       
\usepackage{url}            
\usepackage{booktabs}       
\usepackage{nicefrac}       
\usepackage{microtype}      
\usepackage{bm}
\usepackage{verbatim}
\usepackage{amsthm}
\usepackage{amssymb}        
\usepackage{mydefs}         
\usepackage{algorithm}
\usepackage{algpseudocode}
\usepackage{mathtools}
\usepackage{natbib}
\usepackage{color}
\usepackage{graphicx}
\usepackage{soul}
\DeclarePairedDelimiter\ceil{\lceil}{\rceil}
\DeclarePairedDelimiter\floor{\lfloor}{\rfloor}

\begin{document}

%
\runningtitle{Accelerated Primal-Dual Algorithms for Distributed Optimization over Networks}
\runningauthor{Jinming Xu$^{\dagger,\ddagger}$, Ye Tian$^\ddagger$, Ying Sun$^\ddagger$, and Gesualdo Scutari$^\ddagger$}

%

\twocolumn[

\aistatstitle{Accelerated Primal-Dual Algorithms for Distributed
Smooth \\ Convex Optimization over Networks}

\aistatsauthor{ Jinming Xu$^{\dagger,\ddagger}$~~~~~~~~~~~~~ \And  Ye Tian$^\ddagger$, Ying Sun$^\ddagger$, and Gesualdo Scutari$^\ddagger$}
\aistatsaddress{State Key Lab of Industrial Control Technology~~~~~~~~~~~~~\\  $^\dagger$Zhejiang University~~~~~~~~~~~~~ \And School of Industrial Engineering \\ $^\ddagger$Purdue University } 
]


\begin{abstract}
This paper proposes a novel family of primal-dual-based distributed algorithms for  smooth, convex, multi-agent optimization  over  networks that uses   only gradient  information and gossip communications.  The algorithms can also employ acceleration on the computation and communications.  We provide a unified analysis of their convergence rate, measured in terms of the Bregman distance associated to the saddle point reformation of the distributed optimization problem.    When acceleration is employed, the rate is shown to be  optimal, in the sense that it matches (under the proposed metric) existing complexity lower bounds of distributed algorithms applicable to such a class of problem and using only gradient information and gossip communications. 
Preliminary numerical results on distributed least-square regression problems show that the proposed algorithm compares favorably on existing distributed schemes.\vspace{-0.2cm}
\end{abstract}

\section{Introduction}\vspace{-0.2cm}
We study distributed (smooth) convex    optimization over multi-agent networks, modeled as a  fixed, undirected graph. Agents  aim to cooperatively solve 
\vspace{-0.1cm}
\begin{equation}  
\label{prob:dop}
\min_{x\in\mathbb{R}^d} F(x):=\sum_{i=1}^m f_i(x), 
\end{equation}
where $x\in\mathbb{R}^d$  is the vector of optimization variables, shared among the $m$ agents; and $f_i:\mathbb{R}^d\rightarrow\mathbb{R}$ is the  cost function of agent $i$, assumed to be smooth, convex   and known only to that agent. We are interested in  network architectures that do not have any centralized (master) node handling  the entire optimization process or able to  gather   information from all the other agents in the system (such as master/slave architectures); 
each agent instead   controls a local estimate of the common vector $x$, which is iteratively updated based upon its local gradient and information   received from  its immediate neighbors. 
 This scenario  arises naturally from several   large-scale machine learning applications wherein  the sheer volume and spatial/temporal disparity of scattered data render centralized processing and storage infeasible or inefficient.   


 The focus of this paper is on {\it optimal rate}  decentralized algorithms for Problem (\ref{prob:dop}) that use only gradient information and gossip communications. By {\it optimal} we mean that these algorithms provably achieve  lower complexity bounds for such a class of  problems and oracle decentralized algorithms.  Primal~\citep{Duchi2012_DualAveraging,yuan2016convergence,Jakovetic2014_FastGradMethod,Nedic2014_LipschitzGrad,di2016next,nedich2016achieving,qu2017harnessing,xu2015augmented,sun2019convergence}
  and primal-dual distributed  methods~\citep{shi2015extra,shi2014linear,ling2015dlm,wei2012distributed_admm,chang2015multi} 
  applicable to Problem (\ref{prob:dop}) have been extensively studied in the literature, enjoying different convergence rates.  In general, these rates are not optimal for several reasons: i) the schemes do not employ any acceleration on the local optimization step
   and/or communications; 
  or ii) they do not balance optimally the number of optimization  and communications  steps.  
Optimal rates of first-order distributed algorithms have been recently studied in~\cite{scaman17optimal,scaman2018optimal,sun2018distributed,uribe2018dual,lan2017communication,shamir2014fundamental,arjevani2015communication} for different classes of optimization problems and network topologies; they however are not optimal or applicable to the formulation considered in this paper.\vspace{-0.2cm}

 \noindent \textbf{Related works.}  
  Optimal lower complexity bounds and matching distributed algorithms have been recently investigated in    \cite{scaman17optimal} for smooth strongly convex   functions, in  \cite{scaman2018optimal} for nonsmooth convex functions, and   in \cite{sun2018distributed} for smooth nonconvex functions.  Fully connected networks have been considered in   \cite{shamir2014fundamental,arjevani2015communication}. However, to  our knowledge, no first-order gossip algorithm  is known that achieves  {\it both}  computation  {\it and} communication lower complexity bound for the minimization of   {\it smooth convex} functions over graphs. Attempts of  designing  accelerated distributed algorithms for Problem \eqref{prob:dop} 
can be found in  \cite{li2018sharp,qu2017accelerated,uribe2018dual}   and are briefly discussed next.   The scheme in \cite{qu2017accelerated}  combines the technique of gradient tracking  \citep{di2016next,xu2015augmented,nedich2016achieving} with  Nesterov acceleration of local computations and   achieves an $\epsilon>0$ solution in   $\Upbound{{1}/{\epsilon^{5/7}}}$ gradient and communication steps, 
under the assumption that the solution set of the optimization problem \eqref{prob:dop} is compact. 
 {Algorithm 7 in \cite{uribe2018dual}   is designed for general smooth convex objectives; it reaches an $\epsilon$ solution in $\Upbound{\sqrt{L_f/(\eta\,\epsilon)} \,\log 1/\epsilon}$ outer loops of communications and $\Upbound{\sqrt{L_f/\epsilon } \log 1/\epsilon}$ inner loops of computations (per communication), resulting in an overall gradient evaluations of $\Upbound{L_f/(\epsilon\sqrt{\eta}) \, \log^2 1/\epsilon}$, which do not match existing lower bounds. 
The subsequent  work \citep{li2018sharp} proposes an accelerated  penalty-based  method with increasing penalty values;  the algorithm achieves the lower bound of $\Upbound{\sqrt{{L_f}/{\epsilon}}}$ gradient evaluations but at the cost of an {\it increasing}  number of communications per gradient evaluation (iteration)--namely: $\Upbound{\sqrt{{L_f}/\bracket{\eta\epsilon }}\log{1}/{\epsilon}}$, making it  not optimal in terms of communication steps.}  
 
\noindent \textbf{Summary of the contributions.} We propose  a novel family of primal-dual-based distributed algorithms for Problem (\ref{prob:dop}) that use {\it only gradient} information and gossip communications. The algorithms can also employ acceleration on the computation and communications.  We provide a unified analysis of their convergence rate, measured in terms of the Bregman distance associated to the saddle point reformation of (\ref{prob:dop}). 
When acceleration on both computation and communications is properly designed,  the proposed algorithms are shown to be optimal, in the sense that they match existing complexity lower bounds \citep{li2018sharp}, rewritten in terms of the Bregman distance metric.  
Furthermore, differently from \cite{scaman17optimal,uribe2018dual},
our algorithms do not  require any  information on the Fenchel conjugate of the agents'  functions, which significantly enlarge the class of functions to which provably optimal rate algorithms can be applied to. Hence, we termed our algorithms 
OPTRA (\emph{optimal conjugate-free distributed primal-dual methods}).  Our preliminary numerical results show that OPTRA compares favorably with existing distributed accelerated methods \citep{li2018sharp,qu2017accelerated,uribe2018dual} proposed for Problem (\ref{prob:dop}), which supports our theoretical findings.

 \noindent \textbf{Technical novelties.}
 {While the genesis of OPTRA finds routs in  the primal-dual algorithm \citep{chambolle2011first} and   employs Nesterov acceleration similarly to  \cite{chen2014optimal} (which also builds on \cite{chambolle2011first}),  there are some substantial differences between the proposed distributed algorithms and  the aforementioned  schemes  \citep{chambolle2011first,chen2014optimal}, which are briefly discussed next.  
The scheme in \cite{chambolle2011first} is meant for abstract saddle-point problems and so \cite{chen2014optimal} does; the focus therein is not on distributed optimization. Hence communications over networks are not explicitly accounted. Furthermore,  \cite{chambolle2011first} does not employ any acceleration while \cite{chen2014optimal} accelerates the computation   but lacks of the communication (networking) component (no gossip-based updates are present in \cite[Alg. 2]{chen2014optimal}).  
On the other hand,    OPTRA adopts Nesterov {\it and} Chebyshev acceleration to balance computation and communication,  so that   lower complexity bounds on  both are achieved (in terms of Bregman distance). This is a major novelty with respect to \cite{chambolle2011first,chen2014optimal}. Because of these differences,  the convergence analysis of OPTRA can not be deducted by that of \cite{chambolle2011first,chen2014optimal}; a novel convergence proof is provided, which  shows an  explicit   dependence of the rate on key  network parameters.
 
\noindent \textbf{Notations:}
 We use $\Null{\cdot}$ (resp.\ $\Span{\cdot}$) to denote the null space (resp.\ range space) of the matrix argument. The vector or matrix (with proper dimension) of all ones (resp.\ all zeros) is denoted by   $\ones$ (resp.\ $\zeros$);   $e_i$   denotes the $i$-th canonical vector; and the identity matrix 
  is denoted by $\mathbf{I}$; the dimensions of these vector and matrices will be clear from the context.  
 The inner product between two matrices  $\x,\y$  is defined as $\innprod{\x}{\y}:=\text{trace}(\x,\y)$ while the induced norm is $\norm{\x}:=\norm{\x}_F$; we will use the same notation for vectors, treated as special cases.     Given a positive semidefinite matrix $\mathbf{G}$, we define  $\innprod{\x}{\x'}_\G=\innprod{\G\x}{\x'}~\text{and}~\norm{\x}_\G=\sqrt{\innprod{\G\x}{\x}}.$  
 \vspace{-0.2cm}



\section{Problem formulation}
\label{gen_inst}\vspace{-0.2cm}

\subsection{Distributed optimization over networks}\vspace{-0.2cm}
We study Problem \eqref{prob:dop} under the following assumptions. 
 

\begin{assum}
\label{assum:L_f-smooth_convex} 
(i) Each cost function $f_i:\mathbb{R}^d\to \mathbb{R}$ is convex and $L_{f_i}$-smooth and 
(ii) Problem \eqref{prob:dop} has a solution.  {Define $L_f := \max_{i=1}^m L_{f_i}$}
\end{assum}
\noindent \textbf{Network  model}
Agents are embedded in a communication network, modeled as  an undirected graph $\Gh=(\Eg,\Vx)$, where $\Vx$ is the set of vertices--the agents--and $\Eg$ is the set of edges; $\{i,j\}\in\Eg$ if there is   a communication link between agent $i$ and agent $j$. We assume that the graph has no self-loops, i.e., $\{i,i\}\notin \Eg$.
 We use $\Nb_i:=\{j|\{i,j\}\in \Eg\}$ to denote the set of    neighbors of agent $i$.
\begin{dfn}[Graph Induced Matrix]
 {The symmetric matrix $\mathbf{S}=[s_{ij}]\in\mathbb{R}^{m\times m}$ is said to be induced by the graph $\Gh=(\Eg,\Vx)$ if $s_{ij}\neq 0$ only if $i=j$ or $\{i,j\}\in\Eg$. The set of such matrices is denoted by  $\mathcal{W}_\Gh$.}\vspace{-0.2cm}
\end{dfn}
Since we are interested in optimization over  networks with no centralized nodes, we will focus on distributed algorithms whereby   agents communicate with their neighbors using a suitably designed   gossip matrix. Standard assumptions on such matrices are the following. 
\begin{assum}
\label{assum:weight_matrix}
Given the graph $\mathcal{G}$, the gossip matrix $\Lap\in \mathbb{R}^{m\times m}$   satisfies:\vspace{-0.2cm}
\begin{itemize}
\item[(i)] $\Lap\in \mathcal{W}_\Gh$;
\item[(ii)] Positive semi-definiteness: $\Lap\succeq 0$, with $0=\lambda_1\leq\lambda_2\leq\lambda_3\leq...\leq\lambda_m$;
\item[(iii)] Connectivity: $\Null{\Lap}= \Span{\ones}$; \vspace{-0.2cm}
\end{itemize}
where $\{\lambda_i\}_{i=1}^m$ are the eigenvalues of $\Lap$.
\end{assum}
It is not difficult to check   a gossip matrix satisfying Assumption \ref{assum:weight_matrix} always exists if the associated graph is connected; {see, e.g.,~\cite{olfati2007consensus}. Several    gossip matrices have been  considered in the literature; we refer the reader to   \cite{xiao2004fast,nedic2018network} and references therein for specific examples.} \vspace{-0.2cm}  


\subsection{Saddle-point reformulation}

  A standard approach for solving (\ref{prob:dop}) consists in rewriting the optimization problem in the  so-called  consensus optimization form, that is \vspace{-0.2cm}\begin{equation}\label{prob:ocp}
\min_{\x\in\mathbb{R}^{m\times d}} f(\x)+\iota_{\mathcal{C}}(\x),\vspace{-0.2cm}
\end{equation}
where $\x=[x_1,x_2,...,x_m]^\top\in\mathbb{R}^{m\times d}$, with $x_i$ being the local estimate of $x$ owned by  agent $i$; $f(\x):=\sum_{i=1}^mf_i(x_i)$; and  $\iota_\mathcal{C}(\cdot)$ is the indicator function on the consensus space $\mathcal{C}:=\{\mathbf{1}_m x^\top \,|\, x\in\mathbb{R}^d\}$. Note that $\nabla f(\x)=[\nabla f_1(x_1),\nabla f_2(x_2),...,\nabla f_m(x_m)]^\top\in\mathbb{R}^{m\times d}$.



To solve Problem~\eqref{prob:ocp}, we consider the following closely related saddle point formulation
\begin{equation}\label{prob:saddle_point}
\max_{\y\in\mathbb{R}^{m\times d}}\min_{\x\in\mathbb{R}^{m\times d}} \Phi(\x,\y):=f(\x)+\innprod{\y}{\x}-\iota_{\mathcal{C}^\perp}(\y),
\vspace{-0.2cm}  
\end{equation}
where $\mathcal{C}^\perp$ is the space orthogonal to $\mathcal{C}$ and $\Phi(\x,\y)$ is the Lagrangian associated to problem~\eqref{prob:ocp}. By  Assumption~\ref{assum:L_f-smooth_convex},  strong duality holds for \eqref{prob:saddle_point}; hence,  \eqref{prob:saddle_point}  admits a primal-dual optimal solution pair $(\x^\star,\y^\star)\in\mathcal{D}:=\mathbb{R}^{m\times d}\times \mathcal{C}^\perp$ that satisfies the following KKT conditions
\vspace{-0.2cm}  
\begin{subequations}
	\label{eq:kkt_condtion}
	\begin{align}
	(\text{Lagrangian Optimality})~~~~~\y^\star&=-\nabla f(\x^\star), \label{eq:kkt_condtion_y}\\
	(\text{Primal Feasibility})~~~~~\x^\star&\in \mathcal{C}, \label{eq:kkt_condtion_x}
	\end{align}
\end{subequations}
and the saddle-point property $\Phi(\x^\star,\y)\leq\Phi(\x^\star,\y^\star)\leq\Phi(\x,\y^\star),$ for all $(\x,\y)\in\mathcal{D}$. Note that $\x^\star$ solves   Problem~\eqref{prob:ocp} and thus it is also a solution of the original formulation \eqref{prob:dop}~\citep{bertsekas2003convex}. 

Using (\ref{prob:saddle_point}) and (\ref{eq:kkt_condtion}),  one can write 
	\begin{equation}\label{eq:Bregman_Lagrangian_relation}
	\begin{aligned}
	&\Phi(\x,\y^\star)-\Phi(\x^\star,\y)=f(\x)+\innprod{\y^\star}{\x}-f(\x^\star)-\innprod{\y}{\x^\star}\\
	&\overset{\eqref{eq:kkt_condtion}}{=}f(\x)-f(\x^\star)-\innprod{\nabla f(\x^\star)}{\x-\x^\star}\overset{\Delta}{=}G(\x,\x^\star)\geq 0.
	\end{aligned}
	\end{equation} 




where $G(\x,\x^\star)$ is the   Bregman distance. The following properties of $G$ are  instrumental for our develoments (the proof is provided in the supporting material). 
\begin{prop}\label{prop:properties_G}
Let  $\x^\star$ be any optimal solution of~\eqref{prob:ocp}; the following hold for  $G$ defined in \eqref{eq:Bregman_Lagrangian_relation}:\vspace{-0.2cm}
\begin{itemize}\setlength{\itemsep}{0pt}
\item [(a)] $\bar{\x}$ is an optimal solution of \eqref{prob:ocp} if and only if $\bar{\x}\in \mathcal{C} $ and $G(\bar{\x},\x^\star)=0$;
\item [(b)] $G(\x,\bullet)$ is constant over the solution set of \eqref{prob:ocp}. 
\end{itemize}
\end{prop}
\vspace{-0.2cm}
Due to (b), for notational  simplicity, in what follows, we will write  $G(\x)$ for $G(\x,\x^\star)$. 
\begin{rem} 
 In this paper we will use $G$ as metric to assess the (worst-case) convergence rate  of the proposed algorithms as well as to state lower complexity bounds. Note that, since  $f$ is   not assumed to be strictly convex,    $G(\x)=0$ does not imply $\x=\x^\star$, but  it is  only   a necessary condition for $\x$ to be optimal (cf. Proposition~\ref{prop:properties_G}(a)).  Still, $G$ is a valid merit function for both purposes above, as explained next.  First,    $G(\x)>\epsilon$ implies that  $\x$  is $\epsilon$ ``far'' away (in the $G$-measure) from {\it any}  optimal solution of \eqref{prob:ocp}; hence, a lower bound in terms of $G$ is an  informative measure. Furthermore, when it comes to the convergence rate analysis of distributed algorithms,  Proposition~\ref{prop:properties_G}-(a) legitimates the use of (the decay rate of) $G$ along the agents' iterates $\{\x^k\}_{k=0}^\infty$, as the distance of $\x^k$  {from $\mathcal{C}$ is proved to be vanishing--see Sec. \ref{sec:optimal_algo}.}   
\end{rem}\vspace{-0.3cm}

\section{Lower Complexity Bounds}\vspace{-0.2cm}

 {We recall here  existing lower complexity bounds for decentralized first-order schemes belonging to the same oracle class of the distributed algorithms we are going to introduce.} The difference from the literature is that we will write such bounds in terms of the Bregman distance $G$. We begin introducing the distributed oracle model, followed by the lower  complexity bound.  \vspace{-0.2cm}



\subsection{Decentralized first-order oracle}
\label{sec:oracle}\vspace{-0.2cm}
Given Problem \eqref{prob:dop} over the graph $\mathcal{G}$, we consider distributed algorithms wherein each agent $i$ controls a local variable $x_i\in \mathbb{R}^d$, which is an estimate of the shared optimization variable $x$ in \eqref{prob:dop}. The value of $x_i$ at (continuous) time $t\in \mathbb{R}_+$ is denoted by $x_i^{(t)}$. To update its own variable, each agent $i$: 1) has access to the gradient of its own function--we assume that the time to inquire such a gradient is normalized to one; and 2)  can communicate values (vectors in $\mathbb{R}^d$) to  (some of)  its neighbors $j\in \mathcal{N}_i$--this communication requires a time $\tau_c\in \mathbb{R}_+$ (which may be smaller or greater than one). Each update $x_i^{(t)}$ is generated by the following general \emph{black-box procedure}.

\noindent \textbf{Distributed first-order oracle $\mathcal{A}$:}
  A distributed first order iterative method generates a sequence $\big\{\x^{(t)}\big\}_{t\geq 0}$, with $\x^{(t)}\triangleq [x_1^{(t)},\ldots ,x_m^{(t)}]$, such that
\begin{equation}\label{eq:oracle}\begin{aligned}
x^{(t)}_i\in & \,\,\underset{\textbf{\small local communication}}{\underbrace{\Span{x^{(s)}_j\,|\,j\in\Nb_i\,\text{ and }\,0\leq s<t-\tau_c}}}\\ &\quad +\underset{\textbf{\small local computation}}{\underbrace{\Span{x_i^{(s)},\nabla f_i(x^{(s)}_i)\,|\,0\leq s<t-1}}},
\end{aligned}\end{equation}
for all $i\in \mathcal{V}$. We made the blanket assumption that each $x_i^0=0$,   without loss of generality. 


The   oracle  \eqref{eq:oracle} allows  each agent to use all the historical values of its local gradients  (local computations) as well as that of the decision variables received from its neighbors (local communications). Furthermore, \eqref{eq:oracle} also captures algorithms employing  multiple rounds of communications (resp. gradient computations) per gradient evaluation (resp.  communication). In the supporting material (Appendix \ref{sec:connections_gradtrack_primaldual}),  we show that the above oracle indeed accounts for most existing distributed algorithms, such as primal-dual methods \citep{shi2015extra} and gradient tracking methods \citep{di2016next,nedich2016achieving,qu2017harnessing,xu2015augmented}.

A similar black-box procedure has been introduced  in \cite{scaman17optimal} for strongly convex instances of   \eqref{prob:dop}. The difference here is that the oracle in \eqref{eq:oracle}  cannot  return the gradient of the conjugate of the $f_i$'s. The reason of considering such  ``less powerful''  methods is that, in practice, it is hard to compute the gradient of conjugate functions. This means that the gossip (dual-based)  methods in \cite{scaman17optimal} do not belong to the oracle considered in this paper. 


\subsection{Lower complexity bounds} \label{sec:LB_complexity}We state   now   lower complexity bounds  in the $G$-metric for the class of algorithms $\mathcal{A}$ applied to Problem \eqref{prob:ocp} [and thus \eqref{prob:dop}] over a connected graph $\mathcal{G}$.   In Section \ref{sec:optimal_algo} we will introduce a primal-dual distributed algorithm that indeed converges to an optimal solution of \eqref{prob:ocp}     driving $G$ to zero  at a rate that matches the  lower complexity bound (see the proofs in the supporting material).

\begin{thm}\label{thm:lowerbound_eigengap}
	Consider Problem \eqref{prob:dop} under Assumption \ref{assum:L_f-smooth_convex} and let $\Gh$  be a connected graph. For any given $\eta\in(0,1]$ and  $L_f>0$, there exists a gossip matrix $\Lap\in \mathcal{W}_\mathcal{G}$  with   eigengap $\eta\triangleq \frac{\lambda_2(\Lap)}{\lambda_m(\Lap)}$, and a set of local cost functions $\{f_i\}_{i=0}^m$, $f_i:\mathbb{R}^d\to \mathbb{R}$, with  {$f(\x)=\sum_i f_i(x_i)$} being $L_f$-smooth 
such that,   for any first-order gossip algorithm in $\mathcal{A}$ using $\Lap$, we have \vspace{-0.2cm}
	\begin{align}\label{eq:lowerbound_eigengap}
	G\big(\x^{(t)}\big) = \Lobound{\frac{  L_f R^2}{\left(\frac{t}{1+\left\lceil \frac{1}{5\sqrt{\eta}} \right\rceil \tau_c}+2\right)^2} +   \frac{R\,\norm{\nabla f(\x^\star)}}{ \frac{t}{1+ \left\lceil \frac{1}{5\sqrt{\eta}} \right\rceil \tau_c }+2}},
	\end{align} for all $t\in \left[0, \frac{d-1}{2}\left({1+\left\lceil \frac{1}{5\sqrt{\eta}} \right\rceil \tau_c}\right)\right]$, where $R\triangleq \|\x^0-\x^\star\|$. 
	Furthermore,\vspace{-0.2cm} 
	\begin{equation}\label{eq:lb_spec}
		\frac{  L_f R^2 }{t \ / \left( 1+\left\lceil \frac{1}{5\sqrt{\eta}} \right\rceil \tau_c \right)} = \Theta \left(   R \norm{\nabla f(\x^\star)}\right).
	\end{equation}
\end{thm}

\begin{col} 
In the setting of  Theorem \ref{thm:lowerbound_eigengap}, the overall time needed by any first-order algorithm in $\mathcal{A}$ using the gossip matrix $\Lap$ to drive $G$ below $\epsilon>0$, 
with $f$ given in Theorem \ref{thm:lowerbound_eigengap}, is  \vspace{-0.2cm}
\begin{equation}\label{eq:Bregman_lowbound}
\Lobound{ \left(1+\frac{1}{\sqrt{\eta}}\tau_c\right)\bracket{   \sqrt{\frac{L_f R^2}{\epsilon} } + \frac{R \norm{\nabla f(\x^\star)}}{\epsilon} } }.\vspace{-0.2cm}
\end{equation}
\end{col}

Notice that, because of \eqref{eq:lb_spec},
  the lower bound \eqref{eq:Bregman_lowbound} can be equivalently stated as  \vspace{-0.2cm}
\begin{equation}\label{eq:lb_equiv}
	\Lobound{ \left(1+\frac{1}{\sqrt{\eta}}\tau_c\right)   \sqrt{\frac{L_f R^2}{\epsilon} } }.\vspace{-0.2cm}
\end{equation}

It is not difficult to check that the lower bound in terms of the traditional function-error-based metric (FEM): \begin{equation}\label{eq:FEM-metric}
\max_{i\in\mathcal{V}} (F(x_i)-\min_{x\in \mathbb{R}^d} F(x)) \end{equation} has the same expression as \eqref{eq:lowerbound_eigengap} [and thus (\ref{eq:Bregman_lowbound}) and \eqref{eq:lb_equiv}] up to some constants.  This observation is also reported in \cite{li2018sharp} without proof, and  stated formally  below for completeness (see the proofs the supporting material).

\begin{thm}[Lower bound on the FEM-metric]\label{thm:fem_lb} In the setting of  Theorem \ref{thm:lowerbound_eigengap}, the overall time needed by any first-order algorithm in $\mathcal{A}$ using the gossip matrix $\Lap$ to drive the function-error-based metric, $\max_{i\in\mathcal{V}} (F(x_i)-\min_{x\in \mathbb{R}^d} F(x)),$ below $\epsilon>0$, 
with $f$ given in Theorem \ref{thm:lowerbound_eigengap}, is bounded by (\ref{eq:Bregman_lowbound}) [or, equivalently, by (\ref{eq:lb_equiv})].
\end{thm}
 


\begin{rem}[Balancing computations \& communications]\label{rmk:lower_bound}
The above lower bounds tell us that one cannot reach an $\epsilon$-solution of \eqref{prob:ocp} (measured either in terms of the $G$ or FEM-metrics) in less than $\Upbound{\sqrt{ L_f R^2/{\epsilon}}+ R \| \nabla f (\x^\star)\|/{\epsilon}}$ computing time and $\Upbound{ \tau_c/\sqrt{\eta} \cdot \bracket{\sqrt{ L_f R^2/\epsilon}+R \| \nabla f (\x^\star)\|/\epsilon }}$ communication time for the worst-case problem as stated in Theorem~\ref{thm:lowerbound_eigengap} (see Eq.~\eqref{eq:cost_func_two-agent} in the supporting material for a concrete example). Since the time for a single gradient evaluation has been normalized to one, the former  lower bound corresponds also to the overall number of gradient evaluations while the overall communication steps read  $\Lobound{1/\sqrt{\eta} \cdot \bracket{\sqrt{ L_f R^2/{\epsilon}}+R \| \nabla f (\x^\star)\|/{\epsilon}}}$. 
This sheds light also on the optimal balance between computation and communication: the optimal number of communication steps per gradient evaluations is $\lceil 1/\sqrt{\eta}\rceil$ (in the worst case). 
In the next section,   we introduce a distributed, gossip-based algorithm that achieves lower complexity bounds in the  $G$-metric.\vspace{-0.2cm}
\end{rem}
 

\section{Distributed primal-dual algorithms}
\label{sec:optimal_algo}\vspace{-0.2cm}
\subsection{A general primal-dual scheme}\vspace{-0.2cm} 
A gamut   of primal-dual algorithms has been  proposed in the literature to solve Problem~\eqref{prob:ocp} in a centralized setting; see, e.g.,  \cite{condat2013primal,chambolle2011first} and references therein for  details. Building on  \cite{condat2013primal,chambolle2011first}, here, 
we propose a general primal-dual algorithm to solve the saddle point problem~\eqref{prob:saddle_point} in a {\it distributed} manner. The algorithm reads: given $\x^k$ and $\y^k$ at iteration  $k$,}
\begin{subequations}\label{alg:g-pd-ATC}
\begin{align}
\x^{k+1}&=\mathbf{A}(\x^k-\gamma(\nabla f(\x^k)+\hat{\y}^k)),\label{alg:g-pd-ATC_x}\\
\y^{k+1}&=\y^k+\tau\mathbf{B}\x^{k+1}, \label{alg:g-pd-ATC_y}\\
\hat{\y}^{k+1}&=\y^{k+1}+(\y^{k+1}-\y^k), \label{alg:g-pd-ATC_haty}
\end{align}
\end{subequations}
where $\y^k$ is the dual vector variable;   $\gamma$ and $\tau$ are the primal and dual step-sizes common to all  the agents;  and $\mathbf{A},\mathbf{B}\in\mathbb{R}^{m\times m}$ 
satisfy the  following assumption.
\begin{assum}\label{assum:A_B_weight_matrices}
 The weight matrices  $\mathbf{A},\mathbf{B}\in\mathbb{R}^{m\times m}$ in (\ref{alg:g-pd-ATC}) are such that   \vspace{-0.3cm}
\begin{enumerate}
 \item[(i)]  $\mathbf{A}=\mathbf{A}^\top$, $\zeros\preceq\mathbf{A}\preceq \I$,  and  $\Null{\I-\mathbf{A}}\supseteq\Span{\mathbf{1}}$;   
  \item[(ii)]  $\mathbf{B}=\mathbf{B}^\top$,  $\mathbf{B}\succeq\zeros$, and $\Null{\mathbf{B}}=\Span{\mathbf{1}}$. 
\end{enumerate}
\end{assum}

\begin{rem}
  Several choices for $\mathbf{A}$ and $\mathbf{B}$ satisfying Assumption \ref{assum:A_B_weight_matrices} are possible, 
resulting in a gamut of specific   algorithms, obtained as   instances of  (\ref{alg:g-pd-ATC}). Note that, when $\mathbf{A}$ and $\mathbf{B}$ satisfy also Assumption \ref{assum:weight_matrix}, all these algorithms are implementable over the graph $\mathcal{G}$. Several examples of such distributed algorithms 
are discussed in details in Appendix~\ref{sec:connections_gradtrack_primaldual}. Here, we only mention that  the gradient tracking methods \citep{di2016next,nedich2016achieving,qu2017harnessing,xu2015augmented}  and  primal-dual methods, such as EXTRA~\citep{shi2015extra}, are all special cases of  (\ref{alg:g-pd-ATC}); the former schemes are obtained   setting $\mathbf{A}=\W^2$ and $\mathbf{B}=(\mathbf{\I-\W})^2$, where $\W\in \mathcal{W}_{\mathcal{G}}$ is the weight matrix used by the agents to employ the consensus step; and EXTRA is obtained  setting $\mathbf{A}=\W$ and $\mathbf{B}=\I-\W$.

 We begin studying convergence of the general primal-dual algorithm (\ref{alg:g-pd-ATC}), under the following tuning of the free parameters:\vspace{-0.2cm}
\begin{equation}\label{eq:Alg-general-setting}
\gamma=\frac{\nu}{\nu L_f+ 1},\quad   \tau=\frac{1}{\nu \lambda_m(\mathbf{B})},\quad (1-\gamma L_f)\I-\gamma\tau\mathbf{B} \succeq \zeros, \quad \end{equation} 
  where  
  $\lambda_m(\mathbf{B})$ is the largest eigenvalue of $\mathbf{B}$.  


\end{rem}

 
\begin{thm}\label{thm:sublinear_rate_g-pd}
Consider Problem \eqref{prob:dop} under Assumption \ref{assum:L_f-smooth_convex}.  Given  $(\x^1,\y^1)$, let $\{(\x^k,\y^k)\}_{k=1}^{\infty}$ be the sequence generated by Algorithm \eqref{alg:g-pd-ATC}, under Assumption~\ref{assum:A_B_weight_matrices} and the setting in \eqref{eq:Alg-general-setting}. Define $\bar{\x}^k:=\frac{1}{k-1}\sum_{t=2}^k \x_t$ and $R\triangleq \|\x^1 -\x^\star\|$. 
Then, the following hold: (i)  $\{\x^k\}_{k=0}^{\infty}$   converges to an optimal solution $\x^\star$ of \eqref{prob:ocp} [thus $\x^\star=\mathbf{1} x^\star$, for some solution  $x^\star$ of  \eqref{prob:dop}]; therefore $\lim_{k\to \infty} G(\x^k)= 0$; and (ii)    the number of iterations needed for $G(\bar{\x}^k)$ to go below $\epsilon>0$ is\footnote{We use $\eta(\B)$ to denote the eigengap of $\B$.}
\begin{equation}
\label{eq:rate-Alg-general}
\Upbound{\frac{L_f R^2}{\epsilon}+\frac{1}{\sqrt{\eta(\mathbf{B})}}\frac{R \| \nabla f (\x^\star)\|}{\epsilon}}.
\end{equation}
\end{thm}
\vspace{-0.2cm}
The proof of the theorem can be found in the supporting material. Note that the convergence rate (\ref{eq:rate-Alg-general}) does not match the lower bound   given in Theorem \ref{thm:lowerbound_eigengap}. For instance, consider as concrete example the choice $\mathbf{A}=\mathbf{I}-\mathbf{L}$ and $\mathbf{B}=\mathbf{L}$; and let  $\tau_c\in \mathbb{R}_+$ (resp. 1) be the time  for each agent to perform a single communication to its neighbors (resp. gradient evaluation).  
The time complexity of the  primal-dual algorithm~\eqref{alg:g-pd-ATC} becomes  $$\Upbound{(1+\tau_c) \left(\frac{L_f R^2}{\epsilon}+\frac{1}{\sqrt{\eta}}\frac{R \| \nabla f (\x^\star)\|}{\epsilon}\right)}.\vspace{-0.2cm}$$ 

To match the lower lower bound given in Theorem \ref{thm:lowerbound_eigengap}, our next step is   accelerating the algorithm, both the computational part and the communication step; we leverage     Nesterov acceleration \citep{nesterov2013introductory} for the optimization step while employ  Chebyshev polynomials \citep{auzinger2011iterative}  to accelerate   communications. To provide some insight of our construction, we begin with the former acceleration;  the latter is added in Section~\ref{sec:cheb}.\vspace{-0.2cm} 

\subsection{Accelerated primal-dual algorithms}
We  accelerate the primal-dual algorithm \eqref{alg:g-pd-ATC} as follows: \vspace{-0.3cm}
\begin{subequations}\label{eq:opt_primal-dual}
\begin{align}
\u^{k+1}&=\mathbf{A}(\x^k-\gamma(\nabla f(\x^k)+\hat{\y}^k)),\label{eq:opt_primal-dual_u}\\
\x^{k+1}&=\u^{k+1}+\alpha_k(\u^{k+1}-\u^k),\label{eq:opt_primal-dual_x}\\
\hat{\x}^{k+1}&=\sigma_k\x^{k+1}+(1-\sigma_k)\u^{k+1}\label{eq:opt_primal-dual_hatx}\\
\y^{k+1}&=\y^k+\tau_k \mathbf{B} \hat{\x}^{k+1},\label{eq:opt_primal-dual_y}\\
\hat{\y}^{k+1}&=\y^{k+1}+\beta_k(\y^{k+1}-\y^k), \label{eq:opt_primal-dual_haty}
\end{align}
\end{subequations}
where $\u^k,$ $\hat{\x}^k,$ and $\hat{\y}^k$ are auxiliary variables and $\alpha_k,\sigma_k,\tau_k,\beta_k$ are parameters to be properly chosen.  Roughly speaking,   \eqref{eq:opt_primal-dual_u},~\eqref{eq:opt_primal-dual_y} and \eqref{eq:opt_primal-dual_haty} are the standard primal-dual steps while \eqref{eq:opt_primal-dual_x} and \eqref{eq:opt_primal-dual_hatx} are the extra steps meant for the acceleration, with  \eqref{eq:opt_primal-dual_x} being the standard Nesterov momentum step and \eqref{eq:opt_primal-dual_hatx} being a correction step. Note that  setting $\alpha_k\equiv 0,\sigma_k\equiv 1,\tau_k\equiv \tau,\beta_k\equiv 1$, the  algorithm   reduces to the   primal-dual method~\eqref{alg:g-pd-ATC}.  
We provide next an instance of \eqref{eq:opt_primal-dual} that is suitable for a distributed implementation.

Let $T$ be the overall number of iterations being carried out. The free  parameters in \eqref{eq:opt_primal-dual} is chosen as follows: 
\begin{equation}\label{eq:alg_param_set}\begin{aligned}
&\mathbf{A}=\I-\Lap/\lambda_m(\Lap), \,\, \mathbf{B}=\Lap/\lambda_m(\Lap),\,\, \gamma=\frac{\nu}{\nu L_f+ T},
\\
&\tau=\frac{1}{\nu T\,\lambda_m(\mathbf{B})}\,\,, \frac{1}{\theta_k}=\frac{1+\sqrt{1+4(\frac{1}{\theta_{k-1}})^2}}{2} \text{ with } \theta_1 = 1, \\
& \sigma_k=\frac{1}{\theta_{k+1}},\,\,\alpha_k=\frac{\theta_{k+1}}{\theta_k}-\theta_{k+1},\,\, \beta_k=\frac{\tau_{k+1}}{\tau_k},\,\,  \tau_k=\frac{\tau}{\theta_k }.\end{aligned}
\end{equation}

 
The resulting scheme is summarized in  Algorithm~\ref{alg:acc_primal-dual}, and its convergence properties are stated in Theorem \ref{thm:opt-pd-upperbound}. 
 {We point out that Theorem \ref{thm:opt-pd-upperbound}, although stated for Algorithm \ref{alg:acc_primal-dual}, can be readily extended to the more general accelerated primal-dual scheme (\ref{eq:opt_primal-dual}),  with other   choices of $\mathbf{A}$ and $\mathbf{B}$  just  satisfying Assumption~\ref{assum:A_B_weight_matrices}.} 

\begin{algorithm}[htpb]
\caption{OPTRA-N}\label{alg:acc_primal-dual} 
 {\bf Input}: number of iterations $T$, Laplacian matrix $\Lap$, $\nu>0$ \\
 {\bf Output}: $(\u^T,\y^T)$ \\
{\bf Initialization}: $y_i^1=0, \forall i\in\Vx$ and $\theta_1=1$ 
\begin{algorithmic}[1]
\State $\hat{\y}^1 = \tau_1 \mathbf{B}  \x^1${, $\u^1 = \x^1$}
\For {$k=1,2,...,T$}
\State compute $\theta_k$ according to~\eqref{eq:alg_param_set},
\For {$\forall i\in\Vx$} in parallel
\State compute the next iterate according to \eqref{eq:opt_primal-dual}, using the tuning as in~\eqref{eq:alg_param_set},
\EndFor
\EndFor
\State {\bf Return} $(\u^T,\y^T)$
\end{algorithmic}
\end{algorithm}

\begin{thm}\label{thm:opt-pd-upperbound}
 {
Consider Problem \eqref{prob:dop} under Assumption \ref{assum:L_f-smooth_convex};   
let $\u^{(t)}$  be the value of the $\u$-vector generated by Algorithm~\ref{alg:acc_primal-dual} at time $t\in \mathbb{R}_+$,  {under Assumptions~\ref{assum:weight_matrix} and \ref{assum:A_B_weight_matrices}}, 
  and  the parameter setting   in~\eqref{eq:alg_param_set}.    
If $\nu=\sqrt{\eta}$, then\vspace{-0.1cm} 
\[
G(\u^{(t)}) = \Upbound{ \frac{L_f R^2}{\left(\frac{t}{1+\tau_c} \right)^2}+\frac{ R^2 + \norm{\nabla f(\x^\star)}^2}{\sqrt{\eta}\frac{t}{1+\tau_c} }}.\vspace{-0.1cm}
\]
If one can set $\nu =  \Upbound{\sqrt{\eta}R /\norm{\nabla f(\x^\star)}},$ the above bound can be improved to\vspace{-0.1cm}
\begin{equation}\label{eq:rate-acc-Nest}
G(\u^{(t)}) = \Upbound{ \frac{L_f R^2}{\left(\frac{t}{1+\tau_c} \right)^2}+\frac{ R \norm{\nabla f(\x^\star)}}{\sqrt{\eta}\frac{t}{1+\tau_c} }}. 
\end{equation}
Furthermore, the consensus error decays as
\begin{align}
	&\norm{\left(\I-\avector\right)\u^{(t)}}=\nonumber\\
&\qquad\Upbound{ \frac{L_f R^2}{\norm{\nabla f(\x^\star)} \left(\frac{t}{1+\tau_c} \right)^2}+\frac{ R^2 + \norm{\nabla f(\x^\star)}^2}{\norm{\nabla f(\x^\star)} \sqrt{\eta}\frac{t}{1+\tau_c} }}.
\end{align}

}
\end{thm}



While the convergence time of Algorithm~\ref{alg:acc_primal-dual} benefits from 
the  Nesterov acceleration of the computation step, it is not optimal in terms of communications (optimal dependence on $\eta$). In fact, when the network is poorly connected, the  second term on the RHS of   (\ref{eq:rate-acc-Nest}) becomes dominant with respect to the first one, and (\ref{eq:rate-acc-Nest}) overall will be  larger than  \eqref{eq:lowerbound_eigengap}.  This is due to the fact that Algorithm~\ref{alg:acc_primal-dual} performs   a   one-consensus-one-gradient update while the lower bound shows an optimal ratio of  $\lceil 1/\sqrt{\eta}\rceil$ in the worst case (cf. Remark \ref{rmk:lower_bound}).  This optimal ratio can be achieved  accelerating also the communication step, as described in the next section.\vspace{-0.2cm}

\subsection{Optimal primal-dual algorithms with Chebyshev acceleration}\label{sec:cheb}\vspace{-0.2cm}
We employ the acceleration of the communication step in Algorithm \ref{alg:acc_primal-dual}  by replacing the  gossip matrix $\Lap$ with $P_K(\Lap)$, where $P_K(\cdot)$ is a polynomial of at most $K$ degree that maximizes the eigengap of $P_K(\Lap)$. This   
leads to a widely used acceleration scheme known as Chebyshev acceleration and the choice $P_K(x) = 1 - T_K(c_1(1-x))/T_K(c_1)$, with  $c_1=(1+\eta(\Lap))/(1-\eta(\Lap))$ and $T_K(\cdot)$, are the Chebyshev polynomials \citep{auzinger2011iterative}.  It is not difficult to check that such a  $P_K(\Lap)$ is still a gossip matrix. Using  in~\eqref{eq:opt_primal-dual} the following setting: \begin{equation}\label{eq:Ceby-acc}\begin{aligned}
&\mathbf{A}=\I-c_2P_K(\Lap),\,\,\mathbf{B}=P_K(\Lap), \,\,K=\left\lceil 1/\sqrt{\eta(\Lap)}\right\rceil,
\, \text{with}\\ & c_2 =  { \bracket{1+2\frac{c_0^K}{(1+c_0^{2K})}}}^{-1}, \,\, c_0 = \frac{1-\sqrt{\eta(\Lap)}}{1+\sqrt{\eta(\Lap)}}, \end{aligned}\end{equation} 
leads to the  distributed scheme  described in 
Algorithm~\ref{alg:opt_primal-dual},  whose convergence rate achieves the lower bound \eqref{eq:Bregman_lowbound}, as   proved in Theorem \ref{thm:opt-pd-upperbound-cheby} below. 
 Although  the idea of using  Chebyshev polynomial   has been already used in some (centralized and distributed)   algorithms   in the literature  \citep{auzinger2011iterative,scaman17optimal},  Algorithm~\ref{alg:opt_primal-dual} substantially differs from that of \citet{scaman17optimal}, which assumes strongly-convex cost functions and is not rate-optimal in the setting considered in this paper (cf.~Sec.~\ref{sec_discussion_scaman} in the supporting material for more details).


\begin{thm}\label{thm:opt-pd-upperbound-cheby}
 {
Consider Problem \eqref{prob:dop} under Assumption \ref{assum:L_f-smooth_convex}; 
let $\u^{(t)}$   be the value of the $u$-vector generated by Algorithm~\ref{alg:opt_primal-dual} at time $t\in \mathbb{R}_+$,  under Assumptions~\ref{assum:weight_matrix} and \ref{assum:A_B_weight_matrices}, 
the parameter setting    in~\eqref{eq:alg_param_set}, and employing  the Chebyshev acceleration   \eqref{eq:Ceby-acc}.   
If $\nu = 1$, then
\[
G(\u^{(t)}) = \Upbound{ \frac{L_f R^2}{\left(\frac{t}{1+\ceil{\frac{1}{\sqrt{\eta}}}\tau_c} \right)^2}+\frac{ R^2 + \norm{\nabla f(\x^\star)}^2}{\frac{t}{1+\ceil{\frac{1}{\sqrt{\eta}}}\tau_c} }}.
\]
If one can set  $\nu = \Upbound{{R}/{\norm{\nabla f(\x^\star)}}},$ the above bound can be improved to\vspace{-0.3cm}
\[
G(\u^{(t)}) = \Upbound{ \frac{L_f R^2}{\left(\frac{t}{1+ \ceil{\frac{1}{\sqrt{\eta}}} \tau_c} \right)^2}+\frac{ R \norm{\nabla f(\x^\star)}}{\frac{t}{1+ \ceil{\frac{1}{\sqrt{\eta}}}\tau_c} }}.
\]
Furthermore, the consensus error $\norm{(\I-\avector)\u^{(t)}}$ decays as
\[
\Upbound{ \frac{L_f R^2}{\norm{\nabla f(\x^\star)} \left(\frac{t}{1+ \ceil{\frac{1}{\sqrt{\eta}}} \tau_c} \right)^2}+\frac{ R^2 + \norm{\nabla f(\x^\star)}^2}{\norm{\nabla f(\x^\star)} \frac{t}{1+\ceil{\frac{1}{\sqrt{\eta}}} \tau_c} }}.
\]
}
\end{thm}

 
According to Theorem~\ref{thm:opt-pd-upperbound-cheby}, given $\epsilon>0,$ the time needed by the algorithm to  drive $G$ below  $\epsilon>0$  is   
\[
\Upbound{\left(1+ \frac{1}{\sqrt{\eta}}\tau_c\right)\bracket{\sqrt{\frac{L_f R^2}{\epsilon}}+\frac{R \| \nabla f (\x^\star)\|}{\epsilon} }},
\]
matching the lower complexity bound given in \eqref{eq:Bregman_lowbound}.

\begin{algorithm}[!ht]
\caption{OPTRA}\label{alg:opt_primal-dual}
{\bf Input}: number of iterations $T$, Laplacian matrix $\widetilde{\Lap}$, number of inner consensus $K = \ceil*{\frac{1}{\sqrt{\eta(\Lap)}}}$, $c_0 = \frac{1-\sqrt{\eta(\widetilde{\Lap})}}{1+\sqrt{\eta(\widetilde{\Lap})}}$, 
$c_1 = \frac{1+\eta(\widetilde{\Lap})}{1- \eta(\widetilde{\Lap})}$, $c_2 = 1/ \bracket{1+2\frac{c_0^K}{1+c_0^{2K}}  }$, $\tau=\frac{c_2}{\nu T  } $, $\gamma=\frac{\nu}{\nu L_f+ T}$, $\nu>0$. \\
{\bf Initialization}: $\y^1=\zeros $; \qquad  {\bf Preprocessing}: $\Lap = \frac{2}{\lambda_2(\widetilde{\Lap}) + \lambda_n(\widetilde{\Lap})} \widetilde{\Lap}.$\\
{\bf Output}: $(\u^T,\y^T)$
\begin{algorithmic}[1]
\State $\hat{\y}^1 = \tau_1 \cdot \Call{AccGossip}{\x^{1}, \Lap, K},\, \u^1 =\x^1$
\For {$k=1,2,...,T$}
\State $\u^{k+\frac{1}{2}}= \x^k-\gamma\bracket{\nabla f(\x^k)+\hat{\y}^k}$,
\State $\u^{k+1} = \u^{k+\frac{1}{2}} - c_2 \cdot \Call{AccGossip}{\u^{k+\frac{1}{2}}, \Lap, K}$, 
\State $\x^{k+1} =\u^{k+1} + \left( \frac{\theta_{k+1}}{\theta_{k}}-\theta_{k+1} \right) (\u^{k+1}-\u^k)$,
\State $\hat{\x}^{k+1} =\frac{1}{\theta_{k+1}} \x^{k+1}+  \left(1-\frac{1}{\theta_{k+1}} \right)  \u^{k+1}$,
\State $\y^{k+1}=\y^k+\frac{\tau}{\theta_k} \Call{AccGossip}{\hat{\x}^{k+1}, \Lap, K}$, 
\State $\hat{\y}^{k+1} =\y^{k+1}+\frac{\theta_k}{\theta_{k+1}}(\y^{k+1}-\y^k)$,
\EndFor
\State {\bf Return} $(\u^T,\y^T)$.
 \item[]
 \Procedure{AccGossip}{$\x, \Lap, K$}
 \State $a_0 = 1, a_1 = c_1 $
 \State $\z_0 = \x, \z_1 = c_1(\I - \Lap) \x$
 \For {$k = 1$ to $K-1$}
 \State $a_{k+1} = 2c_1 a_k - a_{k-1}$
 \State $\z_{k+1} = 2c_1 (\I- \Lap) \z_{k} - \z_{k-1}$
 \EndFor
 \State \Return $\z_0 - \frac{\z_K}{a_K}$
\EndProcedure
\end{algorithmic}
\end{algorithm}

\vspace{-0.2cm}

Note that the optimality is stated in terms of the G-metric and does not imply that the algorithm is rate optimal also in the FEM-metric (\ref{eq:FEM-metric}), which to date remains an open question. In our experiments (cf. Sec. \ref{sec:simulation})  we observed i) the same behavior of the two errors as a function of the total number of computations and communications; and ii) that Algorithm 2 in fact outperforms existing distributed schemes.  

\vspace{-0.2cm}

\begin{figure*}[ht]
\centering{\includegraphics[scale=0.38]{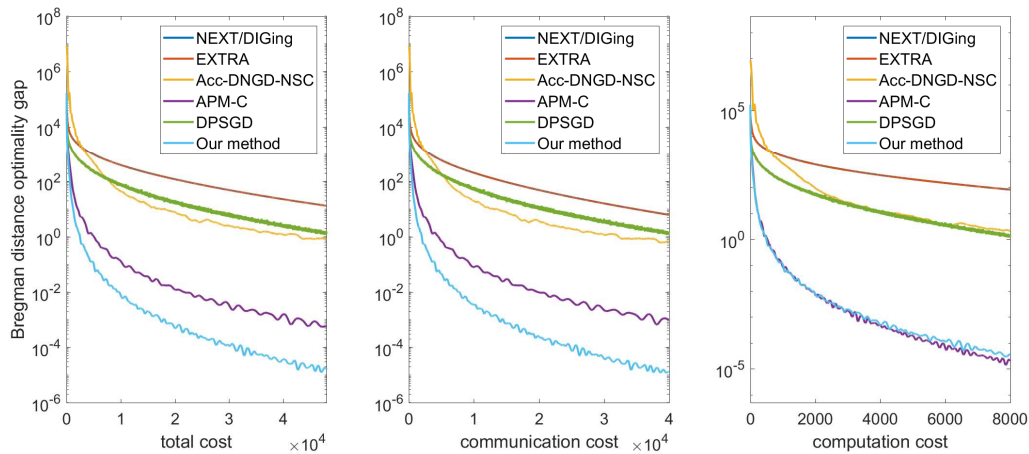}}

\vspace{-0.2cm}
\caption{Comparison of distributed first-order gradient algorithms:    Bregman distance versus  the total cost (left panel), the communication cost (middle panel), and the gradient computation cost (right panel).  The curves of DIGing/NEXT overlap with that of EXTRA.  }\label{fig:supplement_sim}\vspace{-0.4cm}
\end{figure*}

\section{Numerical Results}\label{sec:simulation}\vspace{-0.2cm}
We report here some preliminary numerical results\footnote{Code: \href{https://github.com/YeTian-93/OPTRA}{https://github.com/YeTian-93/OPTRA}.} validating our theoretical findings. We compare the  proposed rate-optimal algorithm--OPTRA--with existing accelerated ones designed for convex smooth problems, i.e., Acc-DNGD-NSC \citep{qu2017accelerated} and APM-C \citep{li2018sharp}. We also included non-accelerated schemes that perform quite well in practice, i.e., i) the  gradient tracking method, NEXT/DIGing \citep{di2016next,nedich2016achieving}; ii) the primal-dual method, EXTRA \citep{shi2015extra};
and iii) the decentralized stochastic gradient method, DPSGD \citep{lian2017can}. 

We tested the above algorithms on a decentralized   
 linear regression problem, in the form 
$
\min_{x \in \mathbb{R}^d} \norm{\mathbf{A}x - \mathbf{b}}^2,$
where $\mathbf{A} = [\mathbf{A}_1; \mathbf{A}_2;\cdots; \mathbf{A}_m] \in \mathbb{R}^{mr \times d}$ and $\mathbf{b} = [\mathbf{b}_1; \mathbf{b}_2;\cdots; \mathbf{b}_m] \in \mathbb{R}^{mr \times 1}$, with $\mathbf{A}_i \in \mathbb{R}^{r\times d}$ and $\mathbf{b}_i \in \mathbb{R}^{r\times 1}$,   $r = 10$, $d = 500$, and $m=20$.  Note that each agent $i$ can only   access  the data $(\mathbf{A}_i,\mathbf{b}_i).$  We generated the matrix  $\mathbf{A}$ of the feature vectors    according to the following  procedure, proposed  in \cite{agarwal2010fast}.  We first generate a random matrix $\mathbf{Z}$ with each entry   i.i.d. drawn from $\mathcal{N}(0,1).$  Using a control parameter $\omega \in [0,1)$, we   generate columns of $\mathbf{A}$ ($\mathbf{M}_{:,i}$ and $\mathbf{M}_{i,:}$  denote the $i$-th column and $i$-th row of a matrix $\mathbf{M}$, respectively) so that the first column is  $\mathbf{A}_{:,1} = \mathbf{Z}_{:,1}/\sqrt{1 - \omega^2}$ and the rest are recursively set as $\mathbf{A}_{:,i} = \omega \mathbf{A}_{:,i-1} + Z_{:,i}$, for $i = 2,\ldots, d$. As   result, each row $\mathbf{A}_{i,:}\in \mathbb{R}^d$ is a Gaussian random vector and its covariance matrix $\Sigma = \text{cov}(\mathbf{A}_{:,i})$ is the identity matrix if $\omega = 0$ and becomes extremely ill-conditioned as $\omega \to 1;$ we set $\omega=0.95$.  Finally we generate $x_0 \in \mathbb{R}^d$ with each entry   i.i.d. drawn  from $\mathcal{N}(0,1)$, and set $\mathbf{b}= \mathbf{A}x_0 + {\bm \xi},$ where each component of the noise ${\bm \xi}$ is  i.i.d. drawn from $\mathcal{N}(0,0.25).$  We simulated a network of  $m = 20$ agents, connected over a communication graph,   generated by the Erd\"{o}s-R\'{e}Tyi model; the probability of having an edge between any
two nodes is set to  0.1.   We calculated $L_f$ from the generated data and used the exact value whenever it is needed.  
 We tuned the free parameters of the simulated algorithms  manually to achieve the best practical performance for each algorithm. This leads to the following choices:   \textbf{i)}  the step size of DIGing/NEXT and EXTRA is set to  $10^{-5};$  \textbf{ii)} for Acc-DNGD-NSC, we used the fixed step-size rule, with   $\eta = 0.005/L_f$ (the one provided in \cite[Th.~5]{qu2017accelerated} is too conservative, resulting in poor practical performance);   \textbf{iii)} for APM-C, we set (see notation therein) $T_k = \ceil{c\cdot ({\log k}/{\sqrt{1-\sigma_2(\mathbf{W}))}}}$, with  $c = 0.2$ and $\beta_0 = 10^4;$  {\textbf{iv)} for DPSGD, we set the step size to $10^{-5}$; at each iteration, the gradient of each agent was computed using $20\%$ of the samples in the  local data set;}  and \textbf{v)} for OPTRA, we set $\nu = 100$ and $K = 2.$ 
 
 Our experiments are reported in   Figure~\ref{fig:supplement_sim}, where we plot the Bregman distance versus the overall number of communications and   computations performed by each agent (left panel), the number of communications (middle panel), and the number of  computations (right panel).  The  time   for local communications and gradient computations using all the local data samples is normalized to one;  for DPSGD, the computation time unit is scaled proportionally to the size of the local mini-batch. The plots in terms of the more traditional FEM-metric are reported in the supporting material, the behavior is consistent with the results in Figure~\ref{fig:supplement_sim}.  

The following comments are in order.  The accelerated schemes and the stochastic algorithm--DPSGD--converge faster than the non-accelerated schemes--NEXT/DIGing, EXTRA (the curves of EXTRA and NEXT/DIGing coincide in all the panels).  In our experiments,    we observed that this gap is quite evident when problems are ill-conditioned. 
From the right panel,  one can see that APM-C performs better than OPTRA and Acc-DNGD-NSC  in terms of the number of  gradient evaluations, which is expected since APM-C employs an increasing number of communication steps per  gradient evaluation. On the other hand,  APM-C suffers from high communication cost (which is evident from the middle panel), making it not competitive with respect to OPTRA in terms of communications.   When both  communication and computation costs are  considered (left panel),  OPTRA outperforms all the  other simulated schemes,   which support our theoretical findings. \vspace{-0.3cm}

\section{Conclusion}\vspace{-0.3cm}
We  studied   distributed gossip first-order methods for smooth convex optimization over networks. 
We  provided a novel  primal-dual distributed algorithm that employs Nesterov acceleration on the optimization step and acceleration of the communication step via Chebyshev  polynomials, balancing thus computation and communication. We also proved that the algorithm achieves the lower complexity bound in the Bregman distance-metric. Preliminary numerical results showed that the proposed scheme outperforms existing distributed algorithms proposed for the same class of problems.  An open question, currently under investigation, is whether the proposed distributed algorithms are rate optimal also in terms of the FEM metric. No such an algorithm is known so far in the literature. 

\section*{Acknowledgments}
This work has been supported by the following grants: NSF of USA under Grants CIF 1719205 and CMMI 1832688; in part by the Army Research Office under Grant W911NF1810238; and in part by NSF of China under Grants U1909207 and 61922058.
\newpage
\bibliographystyle{apa}
\bibliography{reference}


 \newpage 
 
 \onecolumn
 
\appendix
\begin{center}
\LARGE{\textsc{\LARGE Supporting Material -- Appendix}}
\end{center}
This document serves as supporting material of the paper entitled ``Accelerated Primal-Dual Algorithms for Distributed
Smooth Convex Optimization over Networks'' and contains all the proofs of the main results in the paper.

\section{Review of existing distributed algorithms and their connections}
\label{sec:connections_gradtrack_primaldual}
This section shows the generality of the first-order oracle $\mathcal{A}$ in \eqref{eq:oracle}    and the proposed distributed   primal-dual algorithmic framework \eqref{alg:g-pd-ATC} by casting several existing   distributed algorithms in the oracle form \eqref{eq:oracle} and algorithmic form  \eqref{alg:g-pd-ATC}. 
\subsection{Some distributed optimization methods}
\paragraph{Distributed gradient methods} One of the first distributed algorithms for Problem \eqref{prob:dop}  was proposed in the seminal work~\cite{nedic2009distributed} and called Distributed Gradient Algorithm (DGD). DGD employing constant step-size can be written in compact form as:
\begin{equation}\label{alg:grad}
\x^{k+1}=\W\x^k-\gamma\nabla f(\x^k),
\end{equation}
where $\W\in\mathcal{W}_\Gh$.  Defining $\x^{(t_k)}=\x^k$, DGD  can be rewritten in a piece-wise continuous form as 
\begin{equation}\label{alg:grad_vt}
\begin{aligned}
\x^{(t_{k+1})}&=\W\x^{(t_k)}-\gamma\nabla f(\x^{(t_k)}),\\
\x^{(t)}&=\x^{(t_k)},~t_k\leq t< t_{k+1},
\end{aligned}
\end{equation}
which  is an instance of the oracle $\mathcal{A}$.
\paragraph{Distributed gradient tracking methods} 
The distributed gradient tracking algorithm,  first proposed in \cite{di2016next,xu2015augmented} and further analyzed in \cite{nedich2016achieving,qu2017harnessing}, reads
\begin{subequations}\label{alg:gradtracking}
\begin{align}
\x^{k+1}&=\W\x^k-\gamma\y^k \label{alg:gradtracking_x}\\
\y^{k+1}&=\W\y^k+\nabla f(\x^{k+1})-\nabla f(\x^k) \label{alg:gradtracking_y}
\end{align}
\end{subequations}
where $\y_k$ is an auxiliary variable aiming at tracking the gradient of the sum-cost function. The above algorithm is proved to converge at linear rate to a solution of  Problem~\eqref{prob:ocp},  under proper conditions on the stepsize $\gamma$. To show its relationship to the oracle, we first rewrite  \eqref{alg:gradtracking} absorbing the tracking variable  $\y$, which yields
\begin{equation*}
\label{eq:gradtracking_equi_form_x}
\x^{k+2}=2\W\x^{k+1}-\W^2\x^k-\gamma(\nabla f(\x^{k+1})-\nabla f(\x^k)),
\end{equation*}
with $\x^1=\W\x^0-\gamma\nabla f(\x^0)$. It is clear that the gradient tracking algorithm   belongs to the oracle $\mathcal{S}$, as each iteration $k$ only involves the historical neighboring information and local gradients at $k-1$ and $k-2$. 
\paragraph{Distributed primal-dual methods} 
Distributed primal-dual algorithms can be generally written in the following form~\cite{shi2015extra}
\begin{subequations}\label{alg:primal-dual}
\begin{align}
\x^{k+1}&=\W\x^k-\gamma\nabla f(\x^k)-\y^k \label{alg:primal-dual_x}\\
\y^{k+1}&=\y^k+(\I-\W)\x^{k+1} \label{alg:primal-dual_y}
\end{align}
\end{subequations}
where $\y_k$ is the dual variable. When $\y^0=\zeros$, the algorithm \ref{alg:primal-dual} can  solve  problem~\eqref{prob:ocp}. Evaluating \eqref{alg:primal-dual_x} at $k+1$ and substituting it into \eqref{alg:primal-dual_y} yields
\begin{equation}
	\label{eq:primal-dual-reqritten}
\x^{k+2}=2\W\x^{k+1}-\W\x^k-\gamma(\nabla f(\x^{k+1})-\nabla f(\x^k)),
\end{equation}
with $\x^1=\W\x^0-\gamma\W\nabla f(\x^0)$. It is easy to check that \eqref{eq:primal-dual-reqritten}   belongs to the oracle $\mathcal{A}$. \begin{rem}
There are some other distributed algorithms that do not belong to the categories above  such as \cite{chen2012diffusion}. However, using similar arguments as above, one can show that they are instances of the oracle $\mathcal{A}$.
\end{rem}
\subsection{Connections between gradient tracking and primal-dual methods} 
We  reveal here an interesting connection between  primal-dual methods and  gradient tracking based methods. More specifically,  setting in \eqref{alg:g-pd-ATC_x} $\mathbf{A}=\W^2$ and $\mathbf{B}=(\mathbf{\I-\W})^2$, one can easily recover gradient tracking methods from the primal-dual ones. To simplify the presentation, we consider a slightly different form of \eqref{alg:g-pd-ATC_x}, i.e., 
\begin{subequations}\label{alg:g-pd-ATC-simplified}
\begin{align}
\x^{k+1}&=\W^2(\x^k-\gamma(\nabla f(\x^k))+(\I-\W)\y^k,\label{alg:g-pd-ATC-simplified_x}\\
\y^{k+1}&=\y^k+(\I-\W)\x^{k+1}. \label{alg:g-pd-ATC_y}
\end{align}
\end{subequations}
Then, from \eqref{alg:g-pd-ATC-simplified_x}, we have at iteration $k+1$
\[
\x^{k+2}=\W^2\x^{k+1}-\gamma\W^2\nabla f(\x^{k+1})-(\I-\W)\y^{k+1}
\]
Subtracting \eqref{alg:g-pd-ATC-simplified_x} from the above equation we have
\[
\begin{aligned}
\x^{k+2}-\x^{k+1}&=\W^2\x^{k+1}-\W^2\x^k-\gamma\W^2(\nabla f(\x^{k+1})-\nabla f(\x^k))-(\I-\W)(\y^{k+1}-\y^k)\\
&=\W^2\x^{k+1}-\W^2\x^k-\gamma\W^2(\nabla f(\x^{k+1})-\nabla f(\x^k))-(\I-\W)^2\x^{k+1}\\
&=2\W\x^{k+1}-\x^{k+1}-\W^2\x^k-\gamma\W^2(\nabla f(\x^{k+1})-\nabla f(\x^k)).
\end{aligned}
\]
Rearranging terms leads to
\[
\x^{k+2}-\W\x^{k+1}=\W(\x^{k+1}-\W\x^k)-\gamma\W^2(\nabla f(\x^{k+1})-\nabla f(\x^k)).
\]
Let $-\gamma\W\y^k=\x^{k+1}-\W\x^k$ and suppose $\W$ is invertible. Then, we have
\begin{subequations}\label{alg:augdgm}
\begin{align}
\x^{k+1}&=\W(\x^k-\gamma\y^k) \label{alg:augdgm_x}\\
\y^{k+1}&=\W(\y^k+\nabla f(\x^{k+1})-\nabla f(\x^k)) \label{alg:augdgm_y}
\end{align}
\end{subequations}
which is exactly the standard gradient tracking method in the ATC form \citep{di2016next,xu2015augmented}. 

\section{Proof of Proposition \ref{prop:properties_G}}\label{sec:properties_G}
Statement (a) is a direct result of    \cite[Prop. 6.1.1]{bertsekas2003convex}. We prove next statement (b). Suppose that there are two optimal solutions $\x^\star$ and $\widetilde{\x}^\star$ such that
\[
\nabla f(\x^\star),\, \nabla f(\widetilde{\x}^\star) \in\mathcal{C}^\perp, \qquad 
\x^\star,\,\widetilde{\x}^\star \in\mathcal{C}\qquad \text{and} \qquad f(\x^\star)=f(\widetilde{\x}^\star).
\]
Since $G(\x,\x^\star)=f(\x)-f(\x^\star)-\innprod{\nabla f(\x^\star)}{\x-\x^\star}\geq 0$  for all $\x\in\mathbb{R}^{m\times d}$, and $G(\widetilde{\x}^\star,\x^\star)=0$,   $\widetilde{\x}^\star$ is the global minimizer of   $G$. Hence, it must be   $\nabla f(\x^\star)=\nabla f(\widetilde{\x}^\star)$, implying  
\[
\begin{aligned}
G(\x,\x^\star)&=f(\x)-f(\x^\star)-\innprod{\nabla f(\x^\star)}{\x-\x^\star}\\
&=f(\x)-f(\widetilde{\x}^\star)-\innprod{\nabla f(\widetilde{\x}^\star)}{\x-\widetilde{\x}^\star}=G(\x,\widetilde{\x}^\star), \qquad \forall \x\in\mathbb{R}^{m\times d},
\end{aligned}
\]
where we have used the fact that $\innprod{\nabla f(\z)}{\z}=0$ for any optimal solution $\z$.

  \section{Proof of Theorem \ref{thm:lowerbound_eigengap}}\label{proof_LB}
As elaborated  in Section~\ref{sec:oracle}, to study the lower complexity bound of the first order distributed oracle $\mathcal{A}$ solving Problem \eqref{prob:ocp} [and thus \eqref{prob:dop}], one can consider $\epsilon$-solutions (i.e.,   $\bar{\x}\in \mathbb{R}^{m\times d}$ such that $G(\bar{\x})\leq \epsilon$) of the following convex optimization problem:
\begin{equation}\label{prob:saddle_point_x_appendix}
\min_{\x\in\mathbb{R}^{m \times d}} G(\x) =f(\x)-\innprod{\nabla f(\x^\star)}{\x-\x^\star}-f(\x^\star).
\end{equation}
 The proof is based on building a worst-case objective function in \eqref{prob:saddle_point_x_appendix} and network graph for which the lower bound is achieved by the best available gossip, distributed algorithm in the oracle $\mathcal{A}$.  To do so we build on the  cost function first   introduced in~\cite{arjevani2015communication} for a fully connected network and later used for a peer-to-peer network in~\cite{scaman17optimal}, both for smooth strongly convex problems.  Since we use a different metric (the Bregman distance) to define the lower bound  and consider smooth convex problems (not necessarily strongly-convex), the analysis in \cite{scaman17optimal} cannot be readily applied to our setting and an ad-hoc  proof of the theorem is needed.

 The path of our proof is the following: i) We start with a simple network consisting of two agents such that the diameter of the network will not come into play--see Sec. \ref{sec:two-agent-net}; and ii)   then we extend our results to a general network composed by an arbitrary number of agents--see Sec. \ref{sec:pf_thm2}.

\subsection{A simple two-agent network}\label{sec:two-agent-net}
We state the result on the simple two-agent network as the following.
\begin{thm}\label{thm:optim_consensus_lower_bound}
Consider a two-agent network with   cost functions given in~\eqref{eq:cost_func_two-agent}. Let $\seq{\x}{0}{\infty}$ be the sequence generated by any first-order algorithm $\mathcal{A}$. 
Suppose $0\leq k\leq\frac{d-1}{2}$. Then, we have
\[
G(\x^k)=\Lobound{\frac{L_f\norm{\x^0-\x^\star}^2}{(k+1)^2}+\frac{\norm{\x^0-\x^\star}\norm{\nabla f(\x^\star)}}{k+1}}.
\]
\end{thm}
We prove the above result in three steps: i) we construct the hard function in Sec.~\ref{sec:con_hard}, which is the worst-case function for all methods belonging to the oracle $\mathcal{A}$; ii) we introduce some intermediate result in Sec.~\ref{sec:two_agent_inter}, which is related to our specific metric--the Bregman distance $G$,   and iii) building on step i-ii, we derive the lower bound in Sec.~\ref{sec:pf_th_two_agent}.

\subsubsection{Construction of the hard function}\label{sec:con_hard}
 Consider a  network composed of two agents. The idea of the proof of the lower complexity bound relies on splitting the ``hard'' function used by Nesterov to prove the iteration complexity of first-order gradient methods for (centralized) smooth convex problems across the agents\cite[Chapter 2]{nesterov2013introductory}. More specifically,  consider  the following   cost functions for the two agents:
\begin{equation}\label{eq:cost_func_two-agent}
\begin{cases}
f_{1,[k]}(x_1)=\frac{L_f}{8}x_1^\top\mathbf{A}_{1,[k]}x_1-\frac{L_f}{4}e_{1}^\top\, x_1,\\
f_{2,[k]}(x_2)=\frac{L_f}{8}x_2^\top\mathbf{A}_{2,[k]}x_2,
\end{cases}
\end{equation}
where
\begin{equation}
\mathbf{A}_{1,[k]}:=
\begin{bmatrix}
1 &0 &0 &0 &0 &\cdots\\
0 &1 &-1 &0 &0  &\cdots\\
0 &-1 &1 &0 &0  &\cdots\\
0 &0 &0 &1 &-1 &\cdots\\
0 &0 &0 &-1 &1  &\cdots\\
\vdots &\vdots &\vdots &\vdots &\vdots &\ddots
\end{bmatrix},
\qquad
\mathbf{A}_{2,[k]}:=
\begin{bmatrix}
1 &-1 &0 &0 &0  &\cdots\\
-1 &1 &0 &0 &0  &\cdots\\
0 &0 &1 &-1 &0  &\cdots\\
0 &0 &-1 &1 &0  &\cdots\\
0 &0 &0 &0 &1 &\cdots\\
\vdots &\vdots &\vdots &\vdots &\vdots &\ddots
\end{bmatrix}\label{eq:two-splitting-matrices}
\end{equation}
are two $d\times d$ matrices with their leading principal minors of order $k\in[1, d]$ having non-zero block diagonals while the rest being zero. 

The key idea of Nesterov proof for the lower complexity bound of centralized first-order gradient methods consists in designing the ``hardest'' function to be minimized by any method belonging to the oracle. This function was shown to be   such that, at iteration $k$,   all these methods produce a new iterate whereby only the   $k$th component is  updated. The choice of  the two agents' cost functions  in~\eqref{eq:cost_func_two-agent} follows the same rationale:       the structure of $\mathbf{A}_{1,[k]}$ and $\mathbf{A}_{2,[k]}$ is such that none of the two   agents  is  able to make progresses towards optimality, i.e., updating  the next component in their local optimization vector (with odd index for agent $1$ and even index for agent $2$)  just performing local gradient updates and     without communication with each other.  This means that at certain stages a communication between the two agents is necessary for the algorithm to make progresses towards optimality. 
Building on the above   idea, we  begin  establishing  the  lower  complexity  bound for the two-agent network problem in terms of gradient evaluations.

\subsubsection{Intermediate results}\label{sec:two_agent_inter}

Now substituting $f(\x)=f_{1,[k]}(x_1) + f_{2,[k]}(x_2) $ in~\eqref{prob:saddle_point_x_appendix} and   ignoring constants, we obtain
\begin{equation}\label{prob:saddle_point_x_fix_k}
\min_{\x\in\mathbb{R}^{2\times d}} f_{[k]}(\x):= f_{1,[k]}(x_1)+f_{2,[k]}(x_2) - \innprod{[\nabla f_{1,[k]}(x^\star_1), \nabla f_{2,[k]}(x^\star_2)]^\top}{\x}.
\end{equation}
We denote the optimal function value of the above problem as $f_{[k]}^\star.$  It is obvious that, when agents reach consensus, i.e., $x_1=x_2$, the function $f_{[k]}(\x)$ will reduce to the Nesterov's ``hard'' function \cite[Section~2.1.2]{nesterov2013introductory}, for which we have the optimal solution
\[
x^\star_1=x^\star_2=[\underset{\text{the first $k$ components}}{\underbrace{\frac{k}{k+1},\frac{k-1}{k+1},...,\frac{1}{k+1}}},0,\ldots, 0]^\top\in\Span{e_1,e_2,...,e_k},
\]
and it yields 
\begin{align}\label{eq:x_norm}
\norm{\x^\star}^2=\norm{x^\star_1}^2+\norm{x^\star_2}^2\leq\frac{2}{3}(k+1)
\end{align}
and $f_{[k]}^\star=\frac{L_f}{8}(-1+\frac{1}{k+1})$. Also, we have
\begin{equation}\label{eq:grad}
\begin{cases}
\nabla f_{1,[k]}(x^\star_1)=\frac{L_f}{4}(\mathbf{A}_{1,[k]}x^\star_1- e_1)=-\frac{L_f}{4}\frac{1}{k+1} a_{[k]}\\
\nabla f_{2,[k]}(x^\star_2)=\frac{L_f}{4}\mathbf{A}_{2,[k]}x^\star_2=\frac{L_f}{4}\frac{1}{k+1} a_{[k]},
\end{cases}
\end{equation}
where
\[
a_{[k]}=[\underset{\text{$1/-1$ alternates $k$ times}}{\underbrace{1, -1, 1,-1, 1, -1,...}}, 0,\ldots,0]^\top.
\]
Thus, we further have
\begin{align}\label{eq:grad_norm}
\norm{\nabla f (\x^\star)} =\sqrt{\norm{\nabla f_{1,[k]}( x^\star_1)}^2+\norm{\nabla f_{2,[k]}( x^\star_2)}^2}=\sqrt{\frac{2L_f^2 a_{[k]}^\top a_{[k]}}{16(k+1)^2}}=\frac{\sqrt{2k}L_f}{4(k+1)}.
\end{align}
Note that quantities \eqref{eq:x_norm} and \eqref{eq:grad_norm} will be useful later to relate the complexities with $\norm{\x^0-\x^\star}$ and $\norm{\nabla f(\x^\star)}.$  
According to \eqref{eq:grad},  Problem~\eqref{prob:saddle_point_x_fix_k} further becomes
\begin{equation}\label{prob:saddle_point_x_fix_k_specified_y}
\min_{\x\in\mathbb{R}^{2\times d}} f_{[k]}(\x)=f_{1,[k]}( x_1)+f_{2,[k]}( x_2)+\frac{L_f}{4(k+1)}\innprod{ a_{[k]}}{ x_1- x_2}.
\end{equation}
In the following, we study the above problem when the local variables $x_1$ and $x_2$ are restricted to the truncating subspace of $\mathbb{R}^d,$ as a stepping stone to prove Theorem~\ref{thm:optim_consensus_lower_bound}.


Let $\mathbb{R}^{k,d}:= \Span{ e_i\in \mathbb{R}^d\, | \, 1\leq i\leq k}$ denote the subspace composed of vectors whose only first $k$ components are possibly non-zeros and $\mathcal{L}^k:=\Span{\nabla f_i( x^l_i) \,|\, 0\leq l\leq k-1,i\in\Vx}.$  It should be noted that the local cost functions constructed in \eqref{eq:cost_func_two-agent} are dependent on $k$, but hereafter subscripts indicating this dependence are omitted for simplicity.
\begin{lem}[Linear Span]
\label{lem:linear_span_property}
Let $\seq{\x}{0}{\infty}$ be the sequence generated by any distributed first-order algorithm $\mathcal{A}$ with $\x^0=\zeros$. Then, $ x^k_i\in\mathcal{L}^k$ for all $k\geq 0$ and all  $i\in\Vx$. 
\end{lem}
The proof of the above lemma is straightforward, since local communication steps do not change the space spanned by the historical gradient vectors generated over the network.

\begin{lem}\label{lem:nonzeros}
Let $\x^0=\zeros$. For the two-agent problem \eqref{eq:cost_func_two-agent}, we have $\mathcal{L}^k\subseteq\mathbb{R}^{k,d}$.
\end{lem}
\begin{proof}
Since $\x^0=\zeros$, we have $\nabla f_1( x_1^0)=-\frac{L_f}{4}e_1\in\mathbb{R}^{1,d},\nabla f_2( x_2^0)=\zeros\in\mathbb{R}^{1,d}$ and thus $\mathcal{L}^1=\Span{\nabla f_1( x_1^0),\nabla f_2( x_2^0)}\subseteq\mathbb{R}^{1,d}$. Now, let $ x_i^j\in\mathcal{L}^j\subseteq\mathbb{R}^{j,d}$. Without loss of generality, let us assume $j$ is odd. Then, according to the structure of $\nabla f_1$, we have $\nabla f_1( x_1^j)=\frac{L_f}{4}(\mathbf{A}_{1,[k]} x_1^j- e_1)\in\mathbb{R}^{j,d}$, but multiplying $\mathbf{A}_{1,[k]}$ from the left of $ x_1^j\in \mathbb{R}^{j,d}$ will not increase the number of nonzeros to $j+1$.  By contrast, for $\nabla f_2$, we have $\nabla f_2( x_2^j)=\frac{L_f}{4}\mathbf{A}_{2,[k]} x_2^j\in\mathbb{R}^{j+1,d}$ and $\mathbf{A}_{2,[k]}$ is now able to increase the number of non-zeros. Therefore, we have $\mathcal{L}^{j+1}= \mathcal{L}^{j} + \Span{\nabla f_1( x_1^j),\nabla f_2( x_2^j)}\subseteq \mathbb{R}^{j+1,d}$ and we can complete the proof by induction.
\end{proof}

\begin{lem}\label{lem:best_optimum_k_nonzeros}
Consider Problem~\eqref{prob:saddle_point_x_fix_k_specified_y}. Let $f^\star_{[k,j]}:=\min_{\x_i\in\mathbb{R}^{j,d},\forall i\in\Vx}f_{[k]}(\x)$; we have
\[
f^\star_{[k,j]}=-\frac{L_f}{8}\bracket{\frac{k^2}{(k+1)^2}+\frac{j}{(k+1)^2}}.
\]
\end{lem}
\begin{proof}
Let $ x_i\in\mathbb{R}^{1,d},$ $i\in\Vx$. Then, the cost function in  \eqref{prob:saddle_point_x_fix_k_specified_y} becomes
\[
f_{[k,1]}(\x):=\frac{L_f}{4}[0.5x_{11}^2-x_{11}+\frac{1}{(k+1)}(x_{11}-x_{21})+0.5x_{21}^2]
\]
which attains the optimum $f^\star_{[k,1]}=\frac{L_f}{8}(-\frac{k^2}{(k+1)^2}-\frac{1}{(k+1)^2})$.

Likewise, letting $ x_i\in\mathbb{R}^{2,d},$   $i\in\Vx$, we have
\[
f_{[k,2]}(\x):=\frac{L_f}{4}[0.5x_{11}^2+0.5x_{12}^2-x_{11}-\frac{1}{k+1}(x_{21}-x_{11})+\frac{1}{k+1}(x_{22}-x_{12})+0.5(x_{21}-x_{22})^2]
\]
which yields $f^\star_{[k,2]}=\frac{L_f}{8}(-\frac{k^2}{(k+1)^2}-\frac{2}{(k+1)^2})$. Also, for $\x_i\in\mathbb{R}^{3,d}$, $i\in\Vx$, we have
\[
\begin{aligned}
f_{[k,3]}(\x)&:=\frac{L_f}{4}[0.5x_{11}^2+0.5(x_{12}-x_{13})^2-x_{11}-\frac{1}{k+1}(x_{21}-x_{11})+\frac{1}{k+1}(x_{22}-x_{12})\\
&-\frac{1}{k+1}(x_{23}-x_{13})+0.5(x_{21}-x_{22})^2+0.5x_{23}^2],
\end{aligned}
\]
which gives $f^\star_{[k,3]}=\frac{L_f}{8}(-\frac{k^2}{(k+1)^2}-\frac{3}{(k+1)^2})$.

In fact, by induction, it is not difficult to show that, when $j$ is odd, for $\x_i\in\mathbb{R}^{j,d}$,  $i\in\Vx$, we have
\[
\begin{aligned}
f_{[k,j]}(\x):=\frac{L_f}{4}\bracket{0.5x_{11}^2-\frac{k}{k+1}x_{11}+\sum_{i=1}^{\frac{j-1}{2}}\bracket{0.5(x_{2(2i)}-x_{2(2i-1)})^2-\frac{1}{k+1}(x_{2(2i)}-x_{2(2i-1)}}\right.\\
\left.+0.5x_{2j}^2-\frac{1}{k+1}x_{2j}+
\sum_{i=1}^{\frac{j-1}{2}}\bracket{0.5(x_{1(2i)}-x_{1(2i+1)})^2-\frac{1}{k+1}(x_{2i}-x_{1(2i+1)}}},
\end{aligned}
\]
which yields
\[
f^\star_{[k,j]}=-\frac{L_f}{8}\bracket{\frac{k^2}{(k+1)^2}+\frac{j}{(k+1)^2}}.
\]
When $j$ is even, for $\x_i\in\mathbb{R}^{j,d},$ $i\in\Vx$, we have
\[
\begin{aligned}
f_{[k,j]}(\x)=\frac{L_f}{4}\bracket{0.5x_{11}^2-\frac{k}{k+1}x_{11}+\sum_{i=1}^{\frac{j}{2}}\bracket{0.5(x_{2(2i)}-x_{2(2i-1)})^2-\frac{1}{k+1}(x_{2(2i)}-x_{2(2i-1)}}\right.\\
\left.+0.5x_{1j}^2-\frac{1}{k+1}x_{1j}+
\sum_{i=1}^{\frac{j}{2}-1}\bracket{0.5(x_{1(2i)}-x_{1(2i+1)})^2-\frac{1}{k+1}(x_{2i}-x_{1(2i+1)}}}
\end{aligned}
\]
which also yields
\[
f^\star_{[k,j]}=-\frac{L_f}{8}\bracket{\frac{k^2}{(k+1)^2}+\frac{j}{(k+1)^2}}.
\]

The proof is completed  by combining the   two cases above.
\end{proof}
\subsubsection{Proof of Theorem~\ref{thm:optim_consensus_lower_bound}}\label{sec:pf_th_two_agent}

We can now prove the theorem. Let us fix $k$ and apply the first-order gossip algorithm $\mathcal{A}$ to minimize $f_{[2k+1]}$. Since $\x^0=\zeros$, invoking Lemma~\ref{lem:best_optimum_k_nonzeros}, we have
\[
\begin{aligned}
G(\x^k)&=f_{[2k+1]}(\x^k)-f^\star_{[2k+1]}\geq\min_{\x\in\mathbb{R}^{k,d}} f_{[2k+1]}(\x)-f^\star_{[2k+1]}= f^\star_{[2k+1,k]}-f^\star_{[2k+1]}\\
&\geq\frac{L_f}{8}(1-\frac{1}{2(k+1)}-\frac{(2k+1)^2}{4(k+1)^2}-\frac{k}{4(k+1)^2})\\
&=\frac{L_f}{32(k+1)} = \Theta \left( \frac{L_f\norm{\x^0-\x^\star}^2}{(k+1)^2}+\frac{ \norm{\x^0-\x^\star}\norm{\nabla f(\x^\star)}}{(k+1)} \right), 
\end{aligned}
\]
where the last inequality comes from  the previously developed facts  $\norm{\x^\star}^2 = \Theta \left( k+1 \right)$, $\norm{\nabla f(\x^\star)}= \Theta \left( \frac{L_f}{\sqrt{k+1}} \right)$ and thus
$\norm{\x^0-\x^\star}^2 = \Theta \big(\frac{k+1}{L_f} \norm{\x^0-\x^\star}\norm{\nabla f(\x^\star)} \big)$.
This completes the proof for the two-agent network.
\hfill $\square$
\begin{rem}
The lower bound we develop in Theorem~\ref{thm:optim_consensus_lower_bound} for distributed scenarios has similar structure of that of the  recent paper \cite{ouyang2018lower}, where the lower bound is derived for general equality-constrained problems in centralized scenarios (i.e., $\mathbf{Ax=b}$). Notice that the results and techniques therein can not apply to our distributed setting, as  we require $\mathbf{b=0}$ and $\mathbf{A}\in\mathcal{W}_\Gh$ while the lower bound in \cite{ouyang2018lower} is determined  by a choice of $\mathbf{b}$ and $\mathbf{A}$ that does not meet our requirement.
\end{rem}

\subsection{Proof of Theorem~\ref{thm:lowerbound_eigengap}}\label{sec:pf_thm2}
 {Following the same path of   \cite{scaman17optimal}}, we now extend the above analysis to the general network setting (arbitrary number of agents) by employing a line graph and constructing certain number of pairwise two-agent networks as in~\eqref{eq:cost_func_two-agent} from the left and the right of the line graph, respectively, yielding two subgroups.
Between these two subgroups, we place a number (proportional to the diameter of the network) of agents with zero cost functions to ensure the necessity of communications between the agents in the two subgroups. To prove the time complexity lower bound, we   then leverage the effect of the network by establishing the connection between the diameter of the network and the eigengap of the gossip matrix. 


Let $\eta_n = \frac{1-\cos\left( \frac{\pi}{T} \right)}{1 + \cos\left( \frac{\pi}{T} \right)}.$  For a given $\eta \in (0,1],$ there exists $n\geq 2$ such that $\eta_n \geq \eta > \eta_{n+1}.$  We treat the cases  $n=2$ and $n\geq 3$ separately.
Let us first consider the case   $n\geq 3$. There exists a line graph of $m=n$ agents and associated Laplacian weight matrix   with  eigengap $\eta.$ Now, let us define two subsets of agents as $\mathcal{A}_l = \left\{  i \big\vert 1\leq i\leq \lceil \zeta m\rceil \right\}$ and $\mathcal{A}_r =  \left\{  i \big\vert \lfloor (1-\zeta ) m \rfloor + 1 \leq i\leq m \right\}$, which lie on the left and the right of the line graph, respectively;  the parameter $\zeta \in (0,\frac{1}{2})$ is to be determined.   The distance between the two subsets is thus $d_c \triangleq \lfloor (1-\zeta ) m \rfloor + 1 - \lceil \zeta m\rceil .$ 
The class of local functions is defined as follows
\begin{equation}\label{eq:test_func_line}
f_i=
\begin{cases}
\frac{L_f}{8} x^\top_i\mathbf{A}_{1,[k]} x_i-\frac{L_f}{4} e_1^\top x_i~&\forall i \in \mathcal{A}_l \\
 \frac{L_f}{8} x_i^\top\mathbf{A}_{2,[k]} x_i~& \forall i \in \mathcal{A}_r \\
0~&\text{otherwise}
\end{cases}
\end{equation}
where $\mathbf{A}_{1,[k]},\mathbf{A}_{2,[k]}$ are the two matrices defined in (\ref{eq:two-splitting-matrices}).
Similarly to the two-agent network case (cf. Sec. \ref{sec:two-agent-net}), we have
\[
\norm{\x^\star}^2\leq\frac{m}{3}(k+1),~\norm{\nabla f(\x^\star)} \leq \sqrt{2 (\zeta m +1 )}\frac{\sqrt{k}L_f}{4(k+1)},
\]
and Problem~\eqref{prob:saddle_point_x_appendix} becomes
\begin{equation}
	\min_{\x\in\mathbb{R}^{m\times d}} f_{[k]}(\x)=\sum_{i=1}^{\lceil \zeta m \rceil }f_i( x_i)+f_{m+1-i}( x_{m+1-i})+\frac{L_f}{4(k+1)}\innprod{a_{[k]}}{ x_i- x_{m+1-i}}
	\end{equation}
which further yields
\[
f^\star_{[k]}= \lceil \zeta m \rceil \frac{L_f}{8}(-1+\frac{1}{k+1})\quad \text{and}\quad ~f^\star_{[k,i]}=- \lceil \zeta m \rceil  \frac{L_f}{8}\bracket{\frac{k^2}{(k+1)^2}+\frac{i}{(k+1)^2}}.
\]
Let each row of $\x^k$ belongs to $\mathbb{R}^{k,d}$. Then, since $\x^0=\zeros$, we have
\begin{equation}\label{eq:G}
\begin{aligned}
G(\x^k)&=f_{[2k+1]}(\x^k)-f^\star_{[2k+1]}\geq\min_{x_i \in\mathbb{R}^{k,d}} f_{[2k+1]}(\x)-f^\star_{[2k+1]}= f^\star_{[2k+1,k]}-f^\star_{[2k+1]}\\
&\geq\frac{\zeta mL_f}{8}(1-\frac{1}{2(k+1)}-\frac{(2k+1)^2}{4(k+1)^2}-\frac{k}{4(k+1)^2})\\
&=\frac{\zeta mL_f}{32(k+1)} = \Theta \left( \frac{  L_f\norm{\x^0-\x^\star}^2}{(k+1)^2}+\frac{ \norm{\x^0-\x^\star}\norm{\nabla f(\x^\star)}}{k+1} \right).
\end{aligned}
\end{equation}
Similarly as the two-agent case, one  can verify that $\norm{\x^0-\x^\star}^2 = \Theta \big(\frac{k+1}{L_f} \norm{\x^0-\x^\star}\norm{\nabla f(\x^\star)} \big)$.

To have at least one non-zero element at the $k$th component among the local copies of agents in both of the above two subsets, one must perform at least $k$ local computation steps and $(k-1)d_c$ communication steps.  Thus, we have
\begin{align}\label{eq:comm}
k\leq \floor*{ \frac{t-1}{1+d_c\tau_c}}+1 \leq \frac{t}{1+d_c\tau_c}+1 .
\end{align}
Choosing $\zeta = \frac{1}{32}$, we have 
\begin{align*}
& d_c = \floor*{ (1-\zeta ) m } + 1 - \ceil*{\zeta m} \geq (1-2\zeta) m -1  \\
& = \frac{15}{16}m -1 \overset{(a)}{\geq} \frac{15}{16} \left( \sqrt{\frac{2}{\eta}} -1 \right) -1  \overset{(b)}{\geq} \frac{1}{5\sqrt{\eta}},
\end{align*}
where (a) is due to $\eta > \eta_{m+1} > \frac{2}{(m+1)^2}$ and (b) is due to $\eta \leq \eta_3 = \frac{1}{3}.$  Further, since $d_c$ is an integer, we have $d_c \geq \ceil*{\frac{1}{5\sqrt{\eta}}}.$
Combining \eqref{eq:G} and \eqref{eq:comm} leads to
\begin{align}\label{eq:lb}
G(\x^{(t)}) \geq \Omega \left(\frac{  L_f\norm{\x^0-\x^\star}^2}{(\frac{t}{1+\lceil \frac{1}{5\sqrt{\eta}} \rceil \tau_c}+2)^2}+\frac{ \norm{\x^0-\x^\star}\norm{\nabla f(\x^\star)}}{\frac{t}{1+ \lceil \frac{1}{5\sqrt{\eta}} \rceil \tau_c }+2} \right).
\end{align}

We focus now on the case  $n=2$. Consider a complete graph of $3$ agents with associated Laplacian matrix having   eigengap equal to  $\eta$. The agents'   cost functions are 
\[
f_i=
\begin{cases}
\frac{L_f}{8} x^\top_i\mathbf{A}_{1,[k]}  x_i-\frac{L_f}{4} e_1^\top\x_i~& i =1 \\
 \frac{L_f}{8} x_i^\top\mathbf{A}_{2,[k]} x_i~&  i =2 \\
0~&i=3
\end{cases}
\]
Following similar steps as above, one can show that
\begin{align*}
 G(\x^k) \geq\Lobound{\frac{  L_f\norm{\x^0-\x^\star}^2}{(k+1)^2}+\frac{  \norm{\x^0-\x^\star}\norm{\nabla f(\x^\star)}}{(k+1)}}~\text{with}~k\leq \frac{t}{1+\tau_c} +1 ~ \text{and} ~1\geq \ceil*{\frac{1}{5\sqrt{\eta}}},
\end{align*}
which leads to the same expression of the lower bound as   in \eqref{eq:lb}. This concludes the proofs.

\section{Proof of Theorem \ref{thm:fem_lb} }\label{sec:pf_lb_fem}
The proof follows the similar line of \cite[Section 2.1.2]{nesterov2013introductory}.  We consider the same set of local cost functions as depicted in \eqref{eq:test_func_line}, with the subscript $[k]$ of the $\A$ matrices replaced by $[2k+1]$. Then, it is not difficult to see that $\min_{x\in \mathbb{R}^d} F(x) = f_{[2k+1]}^\star$ and, for any $x\in \mathbb{R}^{k,d},$ we have
\begin{equation}\label{eq:lb_fem}
\begin{aligned}
& F(x) - f_{[2k+1]}^\star =  \min_{y\in \mathbb{R}^{k,d}} F(y) - f_{[2k+1]}^\star = f_{[k]}^\star - f_{[2k+1]}^\star \\
& = \lceil \zeta m\rceil  \frac{L_f}{8}\left(-1 + \frac{1}{k+1} +1 - \frac{1}{2k+1+1} \right)  = \lceil \zeta m\rceil  \frac{L_f}{16} \frac{1}{k+1}  = \Theta \left(\frac{L_f \,m}{k+1} \right).
\end{aligned}
\end{equation}
For the cost functions as mentioned above, one can also verify that (cf.~Appendix~\ref{sec:pf_th_two_agent})
\[
\norm{\x^\star - \x^0}^2 = \Theta \left( m(k+1) \right), \quad \norm{\nabla f(\x^\star)} = \Theta \left(  \frac{\sqrt{m} L_f}{ \sqrt{k+1}} \right),
\]
and thus $\frac{L_f\norm{\x^\star - \x^0}^2}{k+1} = \Theta \left( \norm{\x^\star - \x^0}  \norm{\nabla f(\x^\star)} \right)$. As a result, the RHS of \eqref{eq:lb_fem} can be rewritten as:
\begin{align*}
 \Theta \left(\frac{L_f \norm{\x^\star - \x^0}^2}{(k+1)^2} +\frac{ \norm{\x^\star - \x^0}  \norm{\nabla f(\x^\star)}}{k+1} \right),~\text{or equivalently,}~ \quad \Theta \left(\frac{L_f \norm{\x^\star - \x^0}^2}{(k+1)^2} \right),
\end{align*}
which translate to the following lower bounds in terms of number of iterations, respectively:
$$
\Lobound{    \sqrt{\frac{L_f \norm{\x^0-\x^\star}^2}{\epsilon} } + \frac{\norm{\x^0-\x^\star}\norm{\nabla f(\x^\star)}}{\epsilon} }\quad ~\text{and}~\quad  \Lobound{ \sqrt{\frac{L_f \norm{\x^0-\x^\star}^2}{\epsilon} } }.
$$
The rest of proof follows by the same argument as in Section~\ref{sec:pf_thm2}  to relate $k$ to the absolute time $t$ as well as the eigengap $\eta$ of the network.

\section{Proofs for the Upper Complexity Bounds}
 {
This section is devoted to the proofs of the upper complexity bounds of the proposed algorithms. We begin in Sec.~\ref{sec:ineq} establishing two fundamental inequalities that are valid for all feasible primal-dual solutions of Problem~\eqref{prob:saddle_point}; see Lemma~\ref{lem:fundamental_inequality} and Lemma~\ref{lem:fundamental_inequality_II}. Then, applying these inequalities to a saddle point solution of Problem~\eqref{prob:saddle_point}, we obtain the convergence rate  of Algorithms~\ref{alg:acc_primal-dual} and \ref{alg:opt_primal-dual}  in terms of the Bregman distance, see Sec.~\ref{app:sublinear_rate_g-pd}and Sec.~\ref{proof-Th_Nest_acc} respectively.  Finally in Sec.~\ref{sec:che}, we apply the  analysis of Chebyshev polynomials to show that the eigengap of the communication matrix $\mathbf{B}$, as a polynomial of the gossip matrix, can be upper bounded by a constant, leading to the upper complexity bound that matches the established lower bound.
}


\subsection{Intermediate results}\label{sec:ineq}
\begin{lem}[Fundamental Inequality I]
\label{lem:fundamental_inequality}
Consider Algorithm ~\eqref{eq:opt_primal-dual}. We define $\tau=\frac{1}{\nu T \lambda_m(\mathbf{B})}.$  Then we have
\[
\sigma_k=\frac{1}{\theta_{k+1}},\quad  \alpha_k=\frac{\theta_{k+1}}{\theta_k}-\theta_{k+1}, \quad \beta_k=\frac{\tau_{k+1}}{\tau_k}, \quad \tau_k=\frac{\tau}{\theta_k}.
\]
Suppose Assumptions~\ref{assum:L_f-smooth_convex} and~\ref{assum:A_B_weight_matrices} hold. Then, for any $\x\in\mathbb{R}^{m\times d}$ and  $\y\in\mathcal{C}^\perp$, we have
\[
\begin{aligned}
&\Phi(\u^{k+1},\y)-\Phi(\mathbf{A}\x,\y)+h(\u^{k+\frac{1}{2}})-h(\x)\\
&\leq -\innprod{\y^{k+1}-\y}{\u^{k+1}-\x}-\frac{1}{\gamma}\innprod{\theta_k(\I - \frac{\gamma\tau}{\theta_k^2}\mathbf{B})(\hat{\x}^{k+1}-\hat{\x}^k)}{\u^{k+1}-\mathbf{A}\x}+\frac{L_f}{2}\norm{\u^{k+1}-\x^k}^2,
\end{aligned}
\]
where $h(\cdot)=\frac{1}{2\gamma}\norm{\cdot}^2_{\mathbf{A}-\mathbf{A}^2}$,~$\u^{k+\frac{1}{2}}=\x^k-\gamma(\nabla f(\x^k)+\hat{\y}^k)$ and $L_f=\max_i\{L_{f_i}\}$.
\end{lem}
\begin{proof}

Since $f$ is $L_f$-smooth by Assumption~\ref{assum:L_f-smooth_convex}, we have
\[
f(\u^{k+1})\leq f(\x^k)+\innprod{\nabla f(\x^k)}{\u^{k+1}-\x^k}+\frac{L_f}{2}\norm{\u^{k+1}-\x^k}^2,
\]
and using  $f(\mathbf{A}\x)\geq f(\x^k)+\innprod{\nabla f(\x^k)}{\mathbf{A}\x-\x^k}$, further gives
\begin{equation}
\label{eq:inequality_smooth_f}
f(\u^{k+1})\leq f(\mathbf{A}\x)+\innprod{\nabla f(\x^k)}{\u^{k+1}-\mathbf{A}\x}+\frac{L_f}{2}\norm{\u^{k+1}-\x^k}^2.
\end{equation}
Also, subtracting $\mathbf{A}\u^{k+1}$ from both sides of \eqref{eq:opt_primal-dual_u},
multiplying    \eqref{eq:opt_primal-dual_y} by $\gamma\mathbf{A}$, and adding the obtained two equations while using \eqref{eq:opt_primal-dual_haty} lead to
\begin{align*}
& (\I-\mathbf{A})\u^{k+1} = -\mathbf{A} \bracket{ \u^{k+1}-\x^k+\gamma\bracket{\nabla f(\x^k)+\y^{k+1}}} -\gamma\mathbf{A}\mathbf{B}(\tau_{k-1}\beta_{k-1}\hat{\x}^k-\tau_k\hat{\x}^{k+1}) \\
& \overset{(*)}{=}  -\mathbf{A} \bracket{\u^{k+1}-\x^k+\gamma\bracket{\nabla f(\x^k)+\y^{k+1}}} - \frac{\gamma \tau}{\theta_k}\mathbf{A}\mathbf{B}( \hat{\x}^k-\hat{\x}^{k+1}) ,
\end{align*}
where in $(*)$ we used $\beta_{k-1}=\frac{\tau_k}{\tau_{k-1}},\tau_k=\frac{\tau}{\theta_k}$.  
Notice that, for the above derivation, we implicitly assume that $k\geq 2.$ However, with the definition of $\hat{\x}^1 := \x^1$ and the fact that $ \hat{\y}^1 = \tau_1 \mathbf{B}  \x^1$, we still have $(\I-\mathbf{A})\u^{2}   = -\mathbf{A} \bracket{\u^{2}-\x^1 +\gamma\bracket{\nabla f(\x^1)+\y^{2}}} - \frac{\gamma \tau}{\theta_1}\mathbf{A}\mathbf{B}( \hat{\x}^1-\hat{\x}^{2}) $.

Multiplying $\u^{k+\frac{1}{2}}-\x$ from both sides of the above equation and using the convexity of $h(\cdot)$ and the fact that $\u^{k+1}=\mathbf{A}\u^{k+\frac{1}{2}}$ we obtain
\begin{equation}
\label{eq:inequality_convex_g}
h(\u^{k+\frac{1}{2}})\leq h(\x)-\frac{1}{\gamma}\innprod{\u^{k+1}-\x^k+\gamma\bracket{\nabla f(\x^k)+\y^{k+1}}-\frac{\gamma\tau}{\theta_k}\mathbf{B}(\hat{\x}^{k+1}-\hat{\x}^k)}{\u^{k+1}-\mathbf{A}\x}.
\end{equation}
Since $\sigma_k=\frac{1}{\theta_{k+1}}$ and $\alpha_k=\frac{\theta_{k+1}}{\theta_k}-\theta_{k+1}$, using \eqref{eq:opt_primal-dual_x} and \eqref{eq:opt_primal-dual_hatx} leads to
\begin{equation}
\label{eq:transform_u_to_xhat}
\theta_k\hat{\x}^{k+1}=\u^{k+1}-(1-\theta_k)\u^k
\end{equation}
and
\[
\begin{aligned}
\hat{\x}^{k+1}-\hat{\x}^k&=\frac{1}{\theta_k}(\u^{k+1}-(1-\theta_k)\u^k)-\frac{1}{\theta_{k-1}}(\u^k-(1-\theta_{k-1})\u^{k-1})\\
&=\frac{1}{\theta_k}\u^{k+1}-\frac{1}{\theta_k}\bracket{(1-\theta_k)\u^k+\frac{\theta_k}{\theta_{k-1}}\u^k-\frac{\theta_k}{\theta_{k-1}}(1-\theta_{k-1})\u^{k-1}}\\
&=\frac{1}{\theta_k}\u^{k+1}-\frac{1}{\theta_k}\bracket{\u^k+(\frac{\theta_k}{\theta_{k-1}}-\theta_k)(\u^k-\u^{k-1})}\\
&\overset{\eqref{eq:opt_primal-dual_x}}{=}\frac{1}{\theta_k}(\u^{k+1}-\x^k).
\end{aligned}
\]
We implicitly assumed $k\geq 2$;  still we have $\hat{\x}^2-\hat{\x}^1 = \frac{1}{\theta_k} \left( \u^2 - \x^1 \right)$, recalling that $\hat{\x}^1 = \x^1$.
Thus, \eqref{eq:inequality_convex_g} becomes
\begin{equation}
\label{eq:inequality_convex_g_simplified}
h(\u^{k+\frac{1}{2}})\leq h(\x)-\frac{1}{\gamma}\innprod{\theta_k(\I - \frac{\gamma\tau}{\theta_k^2}\mathbf{B})(\hat{\x}^{k+1}-\hat{\x}^k)+\gamma\bracket{\nabla f(\x^k)+\y^{k+1}}}{\u^{k+1}-\mathbf{A}\x}.
\end{equation}
Combining \eqref{eq:inequality_smooth_f} and \eqref{eq:inequality_convex_g_simplified} yields: for any $\x\in\mathbb{R}^{m\times d}$ and  $\y\in\mathcal{C}^\perp$,
\[
\begin{aligned}
&f(\u^{k+1})+h(\u^{k+\frac{1}{2}})+\innprod{\y}{\u^{k+1}-\mathbf{A}\x}-f(\mathbf{A}\x)-h(\x)\\
&\leq \innprod{-(\y^{k+1}-\y)-\frac{1}{\gamma}\theta_k( \I - \frac{\gamma\tau}{\theta_k^2}\mathbf{B} )(\hat{\x}^{k+1}-\hat{\x}^k)}{\u^{k+1}-\mathbf{A}\x}+\frac{L_f}{2}\norm{\u^{k+1}-\x^k}^2,
\end{aligned}
\]
which, recalling that $\Phi(\x,\y)=f(\x)+\innprod{\y}{\x}$, completes the proof.
\end{proof}

\begin{lem}[Fundamental Inequality II]
\label{lem:fundamental_inequality_II}
In the setting of Lemma~\ref{lem:fundamental_inequality}, let $\frac{1}{\theta_{k-1}^2}-\frac{1-\theta_k}{\theta_k^2}=0$, that is, $\frac{1}{\theta_k}=\frac{1+\sqrt{1+4(\frac{1}{\theta_{k-1}})^2}}{2}$, with $\theta_1=1$ and $(1-\gamma L_f)\I-\frac{\gamma\tau}{\theta_k^2}\mathbf{B} \succeq \zeros$, for all $1\leq k\leq T-1$. Suppose Assumptions~\ref{assum:L_f-smooth_convex} and~\ref{assum:A_B_weight_matrices} hold. Then, for any $\x\in\mathcal{C}, \y\in\mathcal{C}^\perp$, we have  
\[
\begin{aligned}
\Phi(\u^T,\y)-\Phi(\x,\y)&\leq \frac{1}{T^2}\bracket{\frac{2}{\gamma}\norm{\u^1-\x}^2+\frac{2}{\tau\lambda_2(\mathbf{B})}\norm{\y}^2}.
\end{aligned}
\]
where $N$ is the overall number of iterations. 
\end{lem}
\begin{proof}

Applying Lemma~\ref{lem:fundamental_inequality} with $\x\in \mathcal{C}$, we have (note that $\mathbf{A}\x=\x$ by Assumption~\ref{assum:A_B_weight_matrices})
\begin{equation}\label{eq:basic_inequality_k+1_to_opt}
\begin{aligned}
&\Phi(\u^{k+1},\y)-\Phi(\x,\y)+h(\u^{k+\frac{1}{2}})-h(\x)\\
&\leq -\innprod{\y^{k+1}-\y}{\u^{k+1}-\x}-\innprod{\frac{1}{\gamma}\theta_k (\I - \frac{\gamma\tau}{\theta_k^2}\mathbf{B})(\hat{\x}^{k+1}-\hat{\x}^k)}{\u^{k+1}-\x}+\frac{L_f}{2}\norm{\u^{k+1}-\x^k}^2.
\end{aligned}
\end{equation}
Likewise, with $\x=\u^{k-\frac{1}{2}}$ we have
\begin{equation}\label{eq:basic_inequality_k+1_to_k}
\begin{aligned}
&\Phi(\u^{k+1},\y)-\Phi(\u^k,\y)+h(\u^{k+\frac{1}{2}})-h(\u^{k-\frac{1}{2}})\\
&\leq -\innprod{\y^{k+1}-\y}{\u^{k+1}-\u^k}-\innprod{\frac{1}{\gamma}\theta_k (\I - \frac{\gamma\tau}{\theta_k^2} \mathbf{B})(\hat{\x}^{k+1}-\hat{\x}^k)}{\u^{k+1}-\u^k}+\frac{L_f}{2}\norm{\u^{k+1}-\x^k}^2.
\end{aligned}
\end{equation}
Let $V_k=\Phi(\u^k,\y)-\Phi(\x,\y)+h(\u^{k+\frac{1}{2}})-h(\x)$. Then, multiplying \eqref{eq:basic_inequality_k+1_to_k} by $1-\theta_k$ and \eqref{eq:basic_inequality_k+1_to_opt} by $\theta_k$, and combing the obtained equations yield
\begin{equation}
\label{eq:basic_inequality_theta_comb}
\begin{aligned}
&V_{k+1}-(1-\theta_k)V_k\\
&\leq -\innprod{\y^{k+1}-\y}{\u^{k+1}-\theta_k\x-(1-\theta_k)\u^k}\\
&-\frac{1}{\gamma}\innprod{ \theta_k (\I - \frac{\gamma\tau}{\theta_k^2} \mathbf{B})(\hat{\x}^{k+1}-\hat{\x}^k) }{\u^{k+1}-\theta_k\x-(1-\theta_k)\u^k}+\frac{L_f}{2}\norm{\u^{k+1}-\x^k}^2\\
&\overset{\eqref{eq:transform_u_to_xhat}}{=} -\theta_k\innprod{\y^{k+1}-\y}{\hat{\x}^{k+1}-\x}-\frac{\theta_k^2}{\gamma}\innprod{\hat{\x}^{k+1}-\hat{\x}^k}{\hat{\x}^{k+1}-\x}_{\I - \frac{\gamma\tau}{\theta_k^2} \mathbf{B}}+\frac{\theta_k^2L_f}{2}\norm{\hat{\x}^{k+1}-\hat{\x}^k}^2\\
& =  -\frac{\theta_k^2}{\tau}\innprod{\y^{k+1}-\y}{\y^{k+1}-\y^k}_{(\mathbf{B}+\mathbf{J})^{-1}}-\frac{\theta_k^2}{\gamma}\innprod{\hat{\x}^{k+1}-\hat{\x}^k}{\hat{\x}^{k+1}-\x}_{\I - \frac{\gamma\tau}{\theta_k^2} \mathbf{B}}+\frac{\theta_k^2L_f}{2}\norm{\hat{\x}^{k+1}-\hat{\x}^k}^2,
\end{aligned}
\end{equation}
where in the last equality we used    $\ones^\top\y^k=\zeros,\forall k\geq 1$ and the following result (recall  $\mathbf{B}\mathbf{J}=\mathbf{J}\mathbf{B}=\zeros$ and $\y\in\mathcal{C}^\perp$):
\[
\begin{aligned}
&\innprod{\y^{k+1}-\y}{\hat{\x}^{k+1}-\x}\\
&\qquad=\innprod{(\mathbf{B}+\mathbf{J})^{-1}(\mathbf{B}+\mathbf{J})(\y^{k+1}-\y)}{\hat{\x}^{k+1}-\x}\\
&\qquad=\innprod{\mathbf{B}(\mathbf{B}+\mathbf{J})^{-1}(\y^{k+1}-\y)}{\hat{\x}^{k+1}-\x}\\
&\qquad=\innprod{\y^{k+1}-\y}{\mathbf{B}(\hat{\x}^{k+1}-\x)}_{(\mathbf{B}+\mathbf{J})^{-1}}\\
&\qquad\overset{\eqref{eq:opt_primal-dual_y}}{=}\frac{\theta_k}{\tau}\innprod{\y^{k+1}-\y}{\y^{k+1}-\y^k}_{(\mathbf{B}+\mathbf{J})^{-1}}.
\end{aligned}
\]
Dividing $\theta_k^2$ from both sides of \eqref{eq:basic_inequality_theta_comb} leads to
\begin{equation}
\label{eq:basic_inequality_Lyapunov}
\begin{aligned}
&\frac{V_{k+1}}{\theta_k^2}-\frac{1-\theta_k}{\theta_k^2}V_k\\
&\leq -\frac{1}{\gamma}\innprod{\hat{\x}^{k+1}-\hat{\x}^k}{\hat{\x}^{k+1}-\x}_{\I - \frac{\gamma\tau}{\theta_k^2} \mathbf{B}}-\frac{1}{\tau}\innprod{\y^{k+1}-\y}{\y^{k+1}-\y^k}_{(\mathbf{B}+\mathbf{J})^{-1}}+\frac{L_f}{2}\norm{\hat{\x}^{k+1}-\hat{\x}^k}^2\\
& = -\frac{1}{2\gamma}\bracket{\norm{\hat{\x}^{k+1}-\hat{\x}^k}^2_{(1- \gamma L_f )\I - \frac{\gamma\tau}{\theta_k^2} \mathbf{B} }+\norm{\hat{\x}^{k+1}-\x}^2_{\I - \frac{\gamma\tau}{\theta_k^2} \mathbf{B} }-\norm{\hat{\x}^k-\x}^2_{\I - \frac{\gamma\tau}{\theta_k^2} \mathbf{B} }}\\
&~~~-\frac{1}{2\tau}\bracket{\norm{\y^{k+1}-\y^k}^2_{(\mathbf{B}+\mathbf{J})^{-1}}+\norm{\y^{k+1}-\y}^2_{(\mathbf{B}+\mathbf{J})^{-1}}-\norm{\y^k-\y}^2_{(\mathbf{B}+\mathbf{J})^{-1}}},
\end{aligned}
\end{equation}
where in the last equality we used 
\[
\mathbf{2\innprod{a-c}{b-c}_\G=\norm{a-c}^2_\G+\norm{b-c}^2_\G-\norm{a-b}^2_\G},~\forall \mathbf{a,b}\in\mathbb{R}^{m\times d}.
\]
Summing \eqref{eq:basic_inequality_Lyapunov} over $k$ from $1$ to $T-1$ yields
\begin{equation}
\label{eq:basic_inequality_Lyapunov_sum}
\begin{aligned}
&\frac{V_T}{\theta_{T-1}^2}-\frac{1-\theta_1}{\theta_1^2}V_1+\sum_{k=2}^{T-1}\bracket{\frac{1}{\theta_{k-1}^2}-\frac{1-\theta_k}{\theta_k^2}}V_k\\
&\leq -\frac{1}{2\gamma}\sum_{k=1}^{T-1}\norm{\hat{\x}^{k+1}-\hat{\x}^k}^2_{(1- \gamma L_f )\I - \frac{\gamma\tau}{\theta_k^2} \mathbf{B} }-\frac{1}{2\gamma}
\sum_{k=2}^{T-1} \gamma \tau ( \frac{1}{\theta_k^2} - \frac{1}{\theta_{k-1}^2} )\norm{\hat{\x}^k-\x}^2_\mathbf{B}\\
&~~~-\frac{1}{2\gamma}\bracket{\norm{\hat{\x}^T-\x}^2_{\I - \frac{\gamma\tau}{\theta_{T-1}^2 } \mathbf{B}} -\norm{\hat{\x}^1-\x}^2_{\I - \frac{\gamma\tau}{\theta_1^2} \mathbf{B}}}\\
&~~~-\frac{1}{2\tau}\bracket{\sum_{k=1}^{T-1}\norm{\y^{k+1}-\y^k}^2_{(\mathbf{B}+\mathbf{J})^{-1}}+\norm{\y^T-\y}^2_{(\mathbf{B}+\mathbf{J})^{-1}}-\norm{\y^1-\y}^2_{(\mathbf{B}+\mathbf{J})^{-1}}}.
\end{aligned}
\end{equation}

Recalling that $\frac{1}{\theta_{k-1}^2}-\frac{1-\theta_k}{\theta_k^2}=0$  and $\theta_1=1$, by induction it is easy to see that $k+1>\frac{1}{\theta_k}\geq\frac{k+1}{2}$ and thus $\frac{1}{\theta^2_k}-\frac{1}{\theta^2_{k-1}}=\frac{1}{\theta_k}>0$. Then, with $\hat{\x}^1=\x^1 :=\u^1,\y^1:= \zeros$, \eqref{eq:basic_inequality_Lyapunov_sum} can be simplified as
\begin{equation}
\label{eq:basic_inequality_Lyapunov_sum_simplified}
\begin{aligned}
\frac{T^2}{4}V_T+\sum_{i=1}^T\frac{1}{2\tau}\norm{\y^{k+1}-\y^k}^2_{(\mathbf{B}+\mathbf{J})^{-1}}+\frac{1}{2\gamma}\sum_{k=1}^{T-1}\norm{\hat{\x}^{k+1}-\hat{\x}^k}^2_{(1- \gamma L_f )\I - \frac{\gamma\tau}{\theta_k^2} \mathbf{B} }\\
\leq \frac{1}{2\gamma}\norm{\u^1-\x}^2_{\I - \frac{\gamma\tau}{\theta_1^2} \mathbf{B}} +\frac{1}{2\tau}\norm{\y}^2_{(\mathbf{B}+\mathbf{J})^{-1}}.
\end{aligned}
\end{equation}
Since $\rho\left( \left(\mathbf{B} + \mathbf{J} \right)^{-1} \right) = \frac{1}{\lambda_{\text{min}}\left( \mathbf{B} + \mathbf{J} \right)} = \frac{1}{\lambda_2(\mathbf{B})}$, $\mathbf{B}\succeq\zeros$ and $(1-\gamma L_f)\I-\frac{\gamma\tau}{\theta_k^2}\mathbf{B} \succeq \zeros$, we further have
\[
\begin{aligned}
\frac{T^2}{4}V_T\leq \frac{1}{2\gamma}\norm{\u^1-\x}^2 +\frac{1}{2\tau}\frac{\norm{\y}^2}{\lambda_2(\mathbf{B})},
\end{aligned}
\]
which, together with the fact that $V_k\geq\Phi(\u^k,\y)-\Phi(\x,\y)$, completes the proof.
\end{proof}

\subsection{Proof of Theorem~\ref{thm:sublinear_rate_g-pd}}\label{app:sublinear_rate_g-pd}
 {
Note that the primal-dual method~\eqref{alg:g-pd-ATC} is a special case of the update~\eqref{eq:opt_primal-dual} with the setting  $\theta_k \equiv 1, \alpha_k\equiv 0,\sigma_k\equiv 1,\tau_k\equiv \tau,\beta_k\equiv 1$. Furthermore,  $\gamma$ and $\tau$ defined in \eqref{eq:Alg-general-setting} satisfy $(1-\gamma L_f)\I-\gamma\tau\mathbf{B}\geq 0$ and $\x^k\equiv \u^k$.  Invoking~\eqref{eq:basic_inequality_Lyapunov_sum} with these parameter settings and $\x:= \x^\star,\, \y:=\y^\star=-\nabla f(\x^\star),\, \hat{\x}^1 := \x^1, \y^1:=\zeros,$ we have
\[
\begin{aligned}
\sum_{k=2}^{T}(\Phi(\x^k,\y^\star)-\Phi(\x^\star,\y^\star))\leq \frac{1}{2\gamma}\norm{\x^1-\x^\star}^2 +\frac{1}{2\tau}\frac{\norm{\y^\star}^2}{\lambda_2(\mathbf{B})},
\end{aligned}
\]
Let $\bar{\x}^T:=\frac{1}{T-1}\sum_{k=2}^T \x_k$. Using the convexity of $\Phi$, we furhter have
\[
\begin{aligned}
\Phi(\bar{\x}^T,\y^\star)-\Phi(\x^\star,\y^\star))& \leq \frac{1}{T-1}\bracket{\frac{1}{2\gamma}\norm{\x^1-\x^\star}^2 +\frac{1}{2\tau}\frac{\norm{\y^\star}^2}{\lambda_2(\mathbf{B})}} \\
& \leq \frac{1}{T-1} \bracket{\frac{L_f}{2}\norm{\x^1-\x^\star}^2 + \frac{1}{2\nu}\norm{\x^1-\x^\star}^2  + \frac{\nu }{2\eta(\mathbf{B})} \norm{\y^\star}^2}.
\end{aligned}
\]
Setting $\nu = \frac{\sqrt{\eta(\mathbf{B})}\norm{\x^1 - \x^\star}}{\nabla f(\x^\star)}$  yields \eqref{eq:rate-Alg-general}.  Finally, it follows from  \eqref{eq:basic_inequality_Lyapunov_sum} that $\hat{\x}^T=\x^T$ is bounded, for every $T\in \mathbb{N}_+$.  The rest of proof is to show that $\x^T\rightarrow \x^\star$, which follows the standard cluster point analysis, as in the proof of~\cite[Th.~1]{chambolle2011first} (refer also to \cite[Remark 3]{chambolle2016ergodic}).
}

\subsection{Proof of Theorem~\ref{thm:opt-pd-upperbound}}\label{proof-Th_Nest_acc}
Since $\gamma=\frac{\nu}{\nu L_f+T},\tau=\frac{1}{\nu T \lambda_m(\mathbf{B})}$ and $\frac{1}{k+1}<\theta_k<\frac{2}{k+1}$, we have
\begin{align}\label{eq:psd}
(1-\gamma L_f)\I-\frac{\gamma\tau}{\theta_k^2}\mathbf{B} \succeq \zeros,\forall 1\leq k\leq T-1.
\end{align}
Then, invoking Lemma~\ref{lem:fundamental_inequality_II} with $\x=\x^\star,\y=\y^\star=-\nabla f(\x^\star)$ and knowing that $\Phi(\u^k,\y^\star)-\Phi(\x^\star,\y^\star)=G(\u^k)\geq 0$ (cf., the relation~\eqref{eq:Bregman_Lagrangian_relation} in the main text), we obtain
\[
\begin{aligned}
G(\u^T)&\leq \frac{\frac{2}{\gamma}{R}_x+\frac{2}{\tau}\frac{{R}_y}{\lambda_2(\mathbf{B})}}{T^2}=\frac{2(L_f+ T/\nu ){R}_x+2\nu T \lambda_m(\mathbf{B}) \frac{{R}_y}{\lambda_2(\mathbf{B})}}{T^2}\\
&=\frac{2L_f{R}_x}{T^2}+\frac{ \frac{2}{\nu}{R}_x+ 2\nu\frac{{R}_y}{\eta(\mathbf{B})} }{T},
\end{aligned}
\]
where $R_x=\norm{\u^1-\x^\star}^2,R_y=\norm{\nabla f(\x^\star)}^2$. Setting $\nu=\sqrt{\eta(\mathbf{B})}$ we have
\[
G(\u^T)\leq\frac{2L_f}{T^2}R_x+\frac{2}{\sqrt{\eta(\mathbf{B})}T}(R_x+R_y),
\]
which, together with the time ($1+t_c$) needed at each iteration, gives the overall time complexity. 

If we set\footnote{Note that this requires accurate estimates on the ratio of $R_x/R_y$, which, indeed, plays a key role of trade-off parameter balancing  gradient computation steps and communication steps.} $\nu=\sqrt{\frac{\eta(\mathbf{B})R_x}{R_y}}$, we have $$G(\u^T,\x^\star) \leq \frac{2L_f{R}_x}{T^2}+\frac{ 4 \sqrt{{R}_x  {R}_y} }{\sqrt{ \eta(\mathbf{B})}T},$$ which matches the lower bound also with respect to $R_x,R_y$. 

In the following, we show that both consensus error and the absolute value of the objective error will converge at the same rate as the Bregman distance.
Invoking Lemma~\ref{lem:fundamental_inequality_II} with $\x=\x^\star$, $\gamma=\frac{\nu}{\nu L_f+T}$, $\tau=\frac{1}{\nu T \lambda_m(\mathbf{B})}$ and $\nu = \sqrt{\eta(\mathbf{B})}$, we have
\[
\begin{aligned}
f(\u^T)-f(\x^\star)+\innprod{\u^T}{\y}=\Phi(\u^T,\y)-\Phi(\x^\star,\y)&\leq \phi(\norm{\y})
\end{aligned}
\]
where $\phi(\cdot):=\frac{2L_f{R}_x}{T^2}+\frac{ \frac{2}{\sqrt{\eta}}({R}_x+ \bracket{\cdot}^2)}{T}$.

Now, setting $\y=2\frac{\tilde{\u}^T}{\norm{\tilde{\u}^T}}\norm{\y^\star}$ where $\tilde{\u}^T=(\I-\avector)\u^T$, we have
\[
f(\u^T)-f(\x^\star)+2\norm{\y^\star}\norm{\tilde{\u}^T}\leq\phi(2\norm{\y^\star})
\]
Also, since $f(\u^T)-f(\x^\star)+\innprod{\u^T}{\y^\star}\geq 0$, we have $f(\u^T)-f(\x^\star)\geq -\norm{\y^\star}\norm{\tilde{\u}^T}$. Thus, combining the above two inequalities yields
\[
\norm{\tilde{\u}^T}\leq \frac{\phi(2\norm{\y^\star})}{\norm{\y^\star}}~\text{and}~\abs{f(\u^T)-f(\x^\star)}\leq \phi(2\norm{\y^\star}).
\]
\hfill $\square$

\subsection{Proof of Theorem~\ref{thm:opt-pd-upperbound-cheby}}\label{sec:che}
 \label{proof_Th_Cheb_acc}
Following the similar lines in \cite[Theorem~4]{scaman17optimal}, we first consider the normalized Laplacian $\Lap$ has a spectrum in $[1-c_1^{-1}, 1+c_1^{-1}].$  According to \cite{scaman17optimal,auzinger2011iterative}, the Chebyshev polynomail $P_K(x) = 1 - \frac{T_K(c_1(1-x))}{T_K(c_1)}$ is the solution of the following problem
\[
\min_{p \in \mathbb{P}_K, p(0) = 0} \max_{x \in [1-c_1^{-1}, 1+c_1^{-1}]} \abs{p(x){}-1}.
\]
As a result, we have
\begin{align}\label{eq:new_spect}
\max_{x \in [1-c_1^{-1}, 1+c_1^{-1}]} \abs{P_K(x)-1} \leq 2\frac{c_0^K}{1+c_0^{2K}}.
\end{align}
Define $\delta = 2\frac{c_0^K}{1+c_0^{2K}}.$  Since Algorithm~\ref{alg:opt_primal-dual} amounts to an instance of Procedure~\eqref{eq:opt_primal-dual} with $\mathbf{A}=\I-c_2 \cdot P_K(\Lap)$ and $\mathbf{B}=P_K(\Lap)$, its convergence proof follows the same lines as that of Theorem~\ref{thm:opt-pd-upperbound} with the following properties of $P_K(\Lap)$: i) $P_K(\Lap)$ is symmetric; ii) according to \eqref{eq:new_spect}, $\zeros\preceq \I-c_2 \cdot P_K(\Lap) \preceq \I$ and $P_K(\Lap)\succeq \zeros$, and $\Null{P_K(\Lap)} = \mathcal{C};$ iii) The values given for $\gamma$ and $\tau$ in Algorithm~\ref{alg:opt_primal-dual} ensures that $(1-\gamma L_f)\I-\frac{\gamma\tau}{\theta_k^2} P_K(\Lap) \succeq \zeros$, analogus to \eqref{eq:psd}.  Therefore we have

\[
\begin{aligned}
G(\u^T) & \leq \frac{\frac{2}{\gamma}{R}_x+\frac{2}{\tau}\frac{{R}_y}{\lambda_{\text{min}}(P_K(\Lap) + \mathbf{J})}}{T^2} \leq \frac{2(L_f+ T/\nu ){R}_x+ 2 \nu T (1+ \delta) \frac{{R}_y}{1- \delta }}{T^2}\\
&  = \frac{2L_f{R}_x}{T^2}+\frac{ \frac{2}{\nu}R_x + 2\,\nu \,R_y \frac{1+\delta}{1-\delta}  }{ N}
\overset{(*)}{=} \frac{2L_f{R}_x}{T^2}+\frac{ 4 \sqrt{{R}_x  {R}_y} }{ N} \sqrt{ \frac{1+\delta}{1-\delta}},
\end{aligned}
\]
where (*) requires a specified $\nu$.  Finally, we have
\[
\sqrt{ \frac{1+\delta}{1-\delta}} =  \left( 1+ \left( \frac{1-\sqrt{\eta}}{1+\sqrt{\eta}} \right)^K \right) \Big/ \left( 1- \left( \frac{1-\sqrt{\eta}}{1+\sqrt{\eta}} \right)^K \right).
\]
Taking $K = \ceil*{\frac{1}{\sqrt{\eta}}}$, we have
\begin{align*}
& \left( \frac{1-\sqrt{\eta}}{1+\sqrt{\eta}} \right)^{\ceil*{\frac{1}{\sqrt{\eta}}}} = \left( 1- \frac{2\sqrt{\eta}}{1+\sqrt{\eta}} \right)^{\ceil*{\frac{1}{\sqrt{\eta}}}} \leq \left( 1- \frac{2}{1+\lceil \frac{1}{\sqrt{\eta}} \rceil} \right)^{\ceil*{\frac{1}{\sqrt{\eta}}}} \\
& \overset{(*)}{\leq} \left( 1- \frac{1}{\lceil \frac{1}{\sqrt{\eta}} \rceil} \right)^{\ceil*{\frac{1}{\sqrt{\eta}}}} < e^{-1},
\end{align*}
where (*) is due to the fact that $\ceil*{\frac{1}{\sqrt{\eta}}} \geq 1$.
Thus, we have $\sqrt{ \frac{1+\delta}{1-\delta}} \leq \frac{1+e^{-1}}{1 - e^{-1}} \leq 2.5$, which, together with the time $(1+K\tau_c)$ needed at each iteration, gives the time complexity as announced.
\hfill $\square$

\section{Additional numerical results  }\subsection{Supplementary plots for the experiments in Section~\ref{sec:simulation}}
\label{sec:supp_sim_FEM}

This section provides additional numerical results, complementing those reported in the paper (cf. Section~\ref{sec:simulation}). In Figure \ref{fig:linreg-fem} we   plot the FEM-metric (\ref{eq:FEM-metric})  versus the overall number of communications and   computations performed by each agent (left panel), the number of communications (middle panel), and the number of  computations (right panel).  The comparison of the different schemes suggests  to the same conclusions as in Sec.\ref{sec:simulation};  the only exception is that in the FEM-metric, the stochastic algorithm--DPSGD--does not present a significant advantage with respect to  the non-accelerated algorithms--NEXT, DIGing and EXTRA.

\begin{figure*}[ht]
\centering{\includegraphics[scale=0.4]{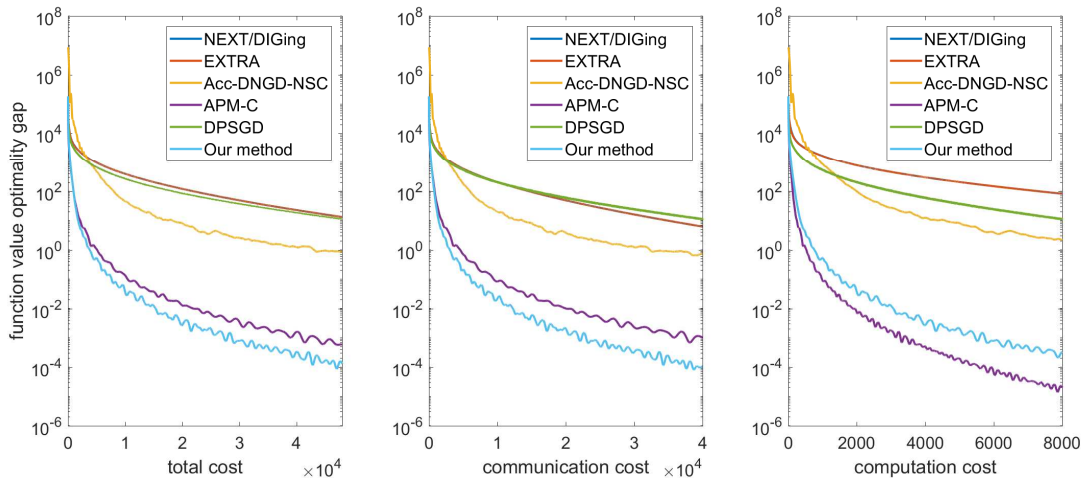}}
\caption{Comparison of distributed first-order gradient algorithms for solving the decentralized linear regression problem in terms of the traditional FEM-metric.}
\label{fig:linreg-fem}
\end{figure*}

\subsection{Decentralized logistic regression}\label{fig:logistic}

\begin{figure*}[ht]
\centering{\includegraphics[scale=0.35]{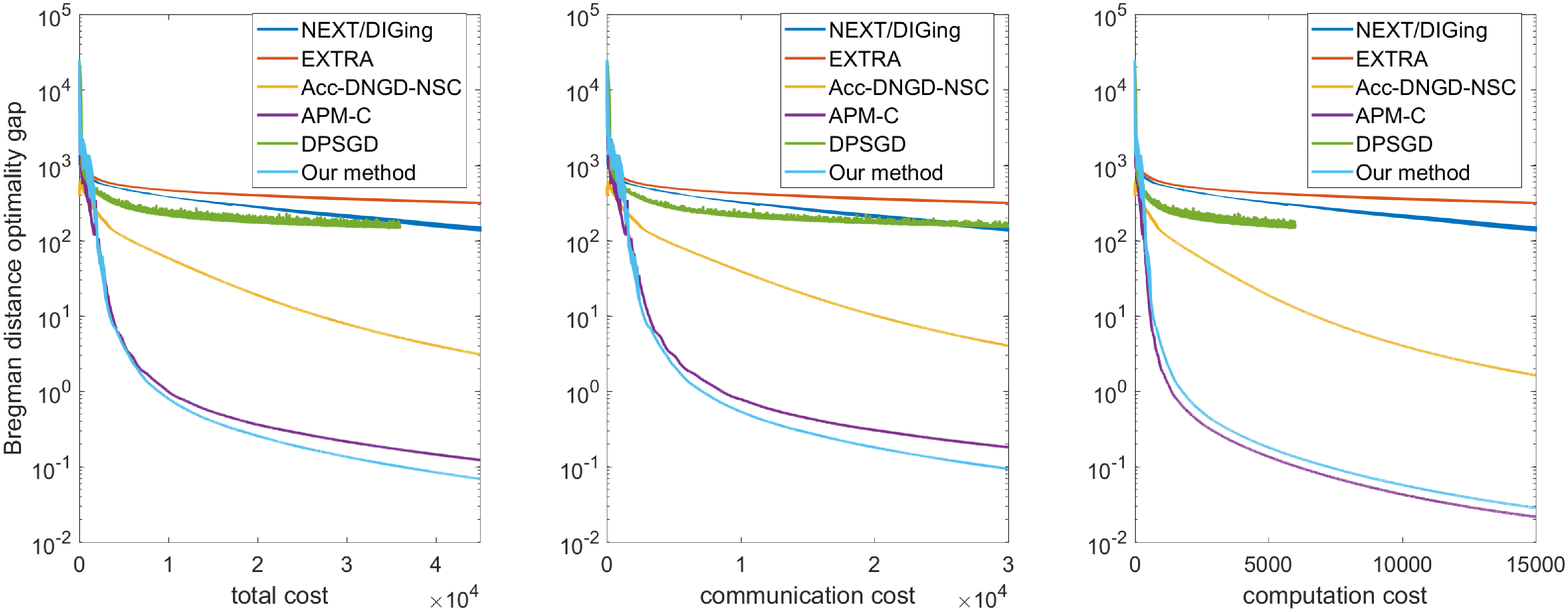}}
\centering{\includegraphics[scale=0.35]{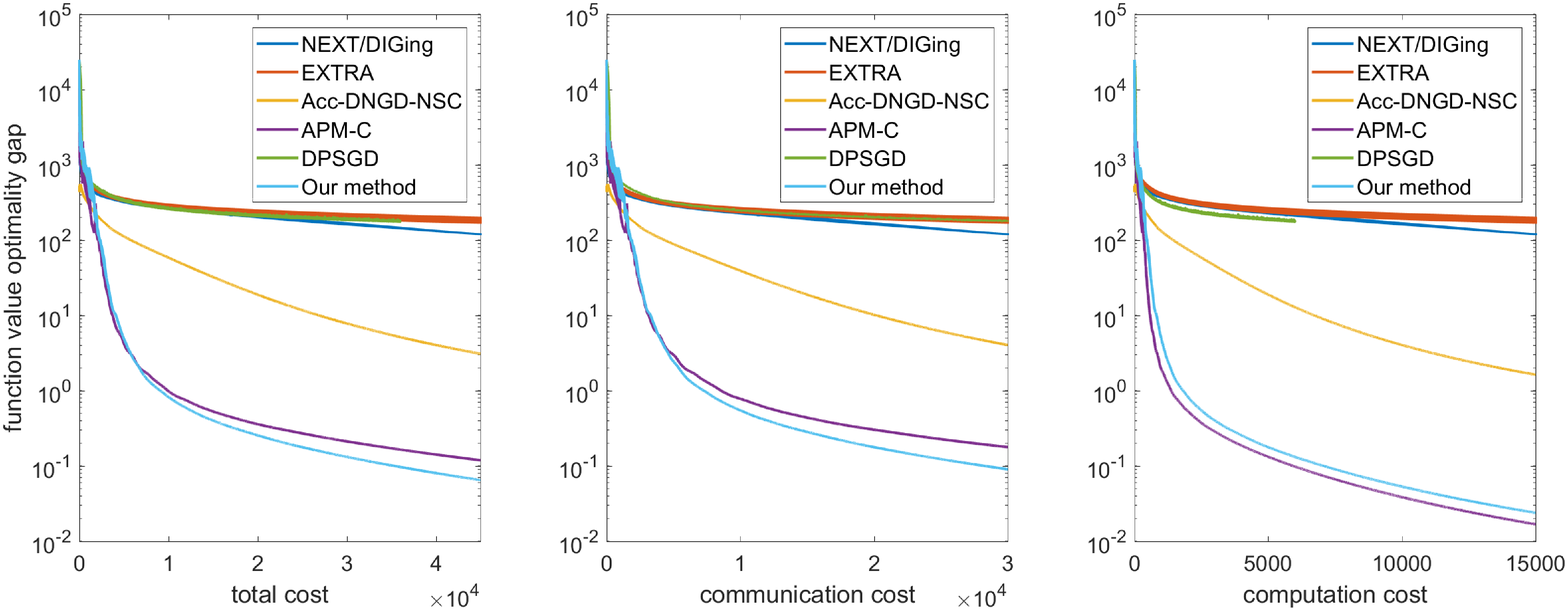}}

\vspace{-0.2cm}
\caption{Comparison of distributed first-order gradient algorithms for solving the decentralized logistic regression problem in terms of both the Bregman distance and the traditional FEM-metric.}
\label{fig:logreg}
\end{figure*}
To further verify the effectiveness of our proposed scheme, we also include a decentralized logistic regression task on the Parkinson's Disease Classification Data Set\footnote{The data set is available at \href{https://archive.ics.uci.edu/ml/datasets/Parkinson\%27s+Disease+Classification}{https://archive.ics.uci.edu/ml/datasets/Parkinson\%27s+Disease+Classification}}.  We preprocess the data by deleting the first column-id number, rescaling feature values to the range $(0,1)$, and changing the label notation from $\{1,0\}$ to $\{1,-1\}$.  We denote the processed data set as $\{(u_i,\, y_i)\}_{i\in \mathcal{D}}$, where  $u_i\in \mathbb{R}^d$ is the feature vector and $y_i\in \{1,-1\}$ is the  label of the $i$-th observation.
We simulated a network of 60 agents, generated by the Erd\"{o}s-R\'{e}Tyi model with the parameter of connection probability as 0.1.  Then, we distributed the data set to all agents evenly, corresponding to a partition of the index set $\mathcal{D}$ across agents as $\mathcal{D} = \cup_{i=1}^{60} \mathcal{D}_i$.  The decentralized logistic regression problem reads
\[
\min_{x\in \mathbb{R}} \sum_{i=1}^m \sum_{j\in \mathcal{D}_i} \log\left(1+\exp(-y_j u_j^\top x)\right).
\]
We estimated $L_f$ for the problem as $L_f=24$ and tuned the free parameters of the simulated algorithms  manually to achieve the best practical performance for each algorithm. This leads to the following choices:   \textbf{i)}  the step size of NEXT/DIGing is set to $0.01$;  \textbf{ii)} the step size of EXTRA is set to  $0.005;$  \textbf{iii)} for Acc-DNGD-NSC, we used the fixed step-size rule, with   $\eta = 0.01/L_f$;   \textbf{iv)} for APM-C, we set (see notation therein) $T_k = \ceil{c\cdot ({\log k}/{\sqrt{1-\sigma_2(\mathbf{W}))}}}$, with  $c = 0.2$ and $\beta_0 = 10^4;$  \textbf{v)} for DPSGD, we set its step size as $0.001$ and the portion of batch size to the full local data set as $20\%$ and for \textbf{vi)} for our algorithm, we set $\nu = 1500$ and $K = 2.$  

The experiment result is reported in Figure~\ref{fig:logreg}.  The first row of panels shows the Bregman distance versus  the total cost (left panel), the communication cost (middle panel), and the gradient computation cost (right panel). The second row plots the FEM-metric versus the same quantities as in the first row.  Both the communication time unit and the computation time unit for a full epoch of local data is set as 1.  For DPSGD, the computation time unit is scaled in proportion to the local batch size.  The only existing algorithm that has a comparable performance with the proposed OPTRA is APM-C.  As discussed in the task of decentralized linear regression, APM-C performs better than OPTRA in terms of the number of gradient computations, while suffers from high communication cost.  In terms of the overall number of communications and computations, OPTRA outperforms all the other simulated schemes under the above setting.

\subsection{Different ratio of communication time versus computation time}
In all the previous experiments, we set both the communication time unit and the computation time unit for a full epoch of local data as 1.  To incorporate scenarios where a full epoch computation of local gradient is much more expensive than one communication process, we re-conducted the previous experiments in the setting where the communication time unit is 1 while the computation time unit for a full epoch of local data is 5.  Note that all the process of data generation and parameters tunings are the same as in the Sec. \ref{fig:logistic}.  The results are reported in Figure~\ref{fig:rescale_linreg} and Figure~\ref{fig:rescale_logreg} respectively for decentralized linear regression problem and the decentralized logistic regression problem.  It can be seen that OPTRA outperforms all the other simulated schemes in terms of the overall number of communications and computations, especially when the communication cost is not negligible.

\begin{figure*}[!ht]
\centering{\includegraphics[scale=0.4]{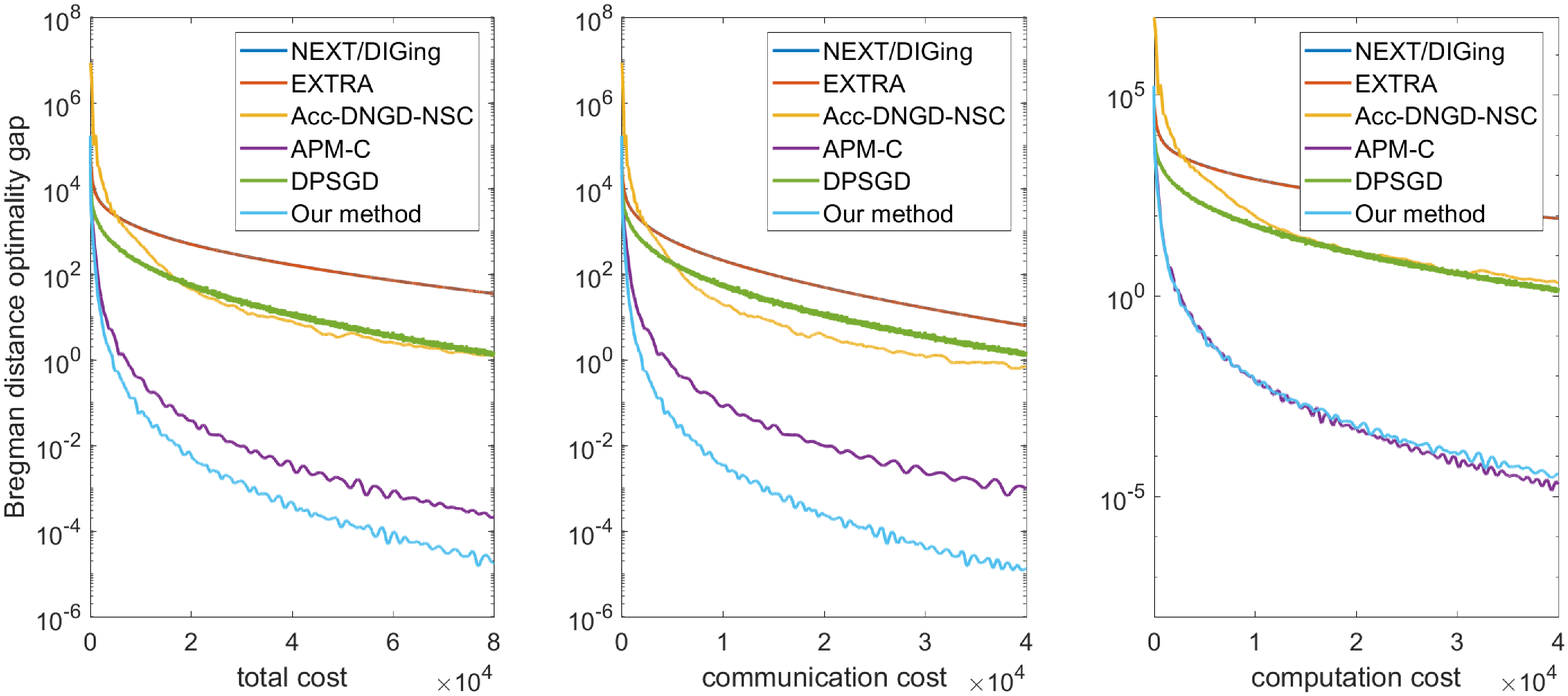}}
\centering{\includegraphics[scale=0.4]{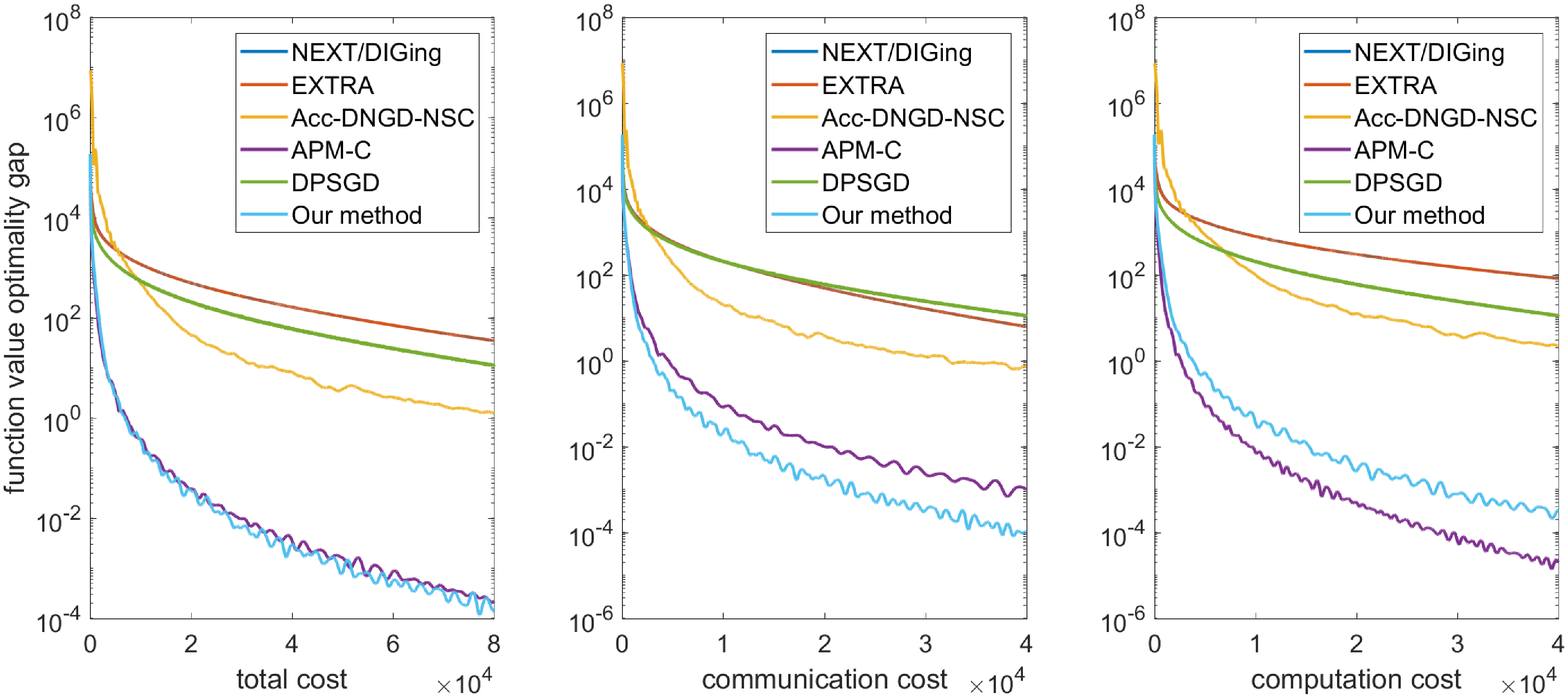}}

\vspace{-0.2cm}
\caption{Comparison for distributed algorithms for solving the decentralized linear regression problem with the communication time unit being ``$1$'' and the computation time unit ``$5$'' for a full epoch of local data.}
\label{fig:rescale_linreg}\vspace{-0.4cm}
\end{figure*}

\begin{figure*}[!ht]
\centering{\includegraphics[scale=0.35]{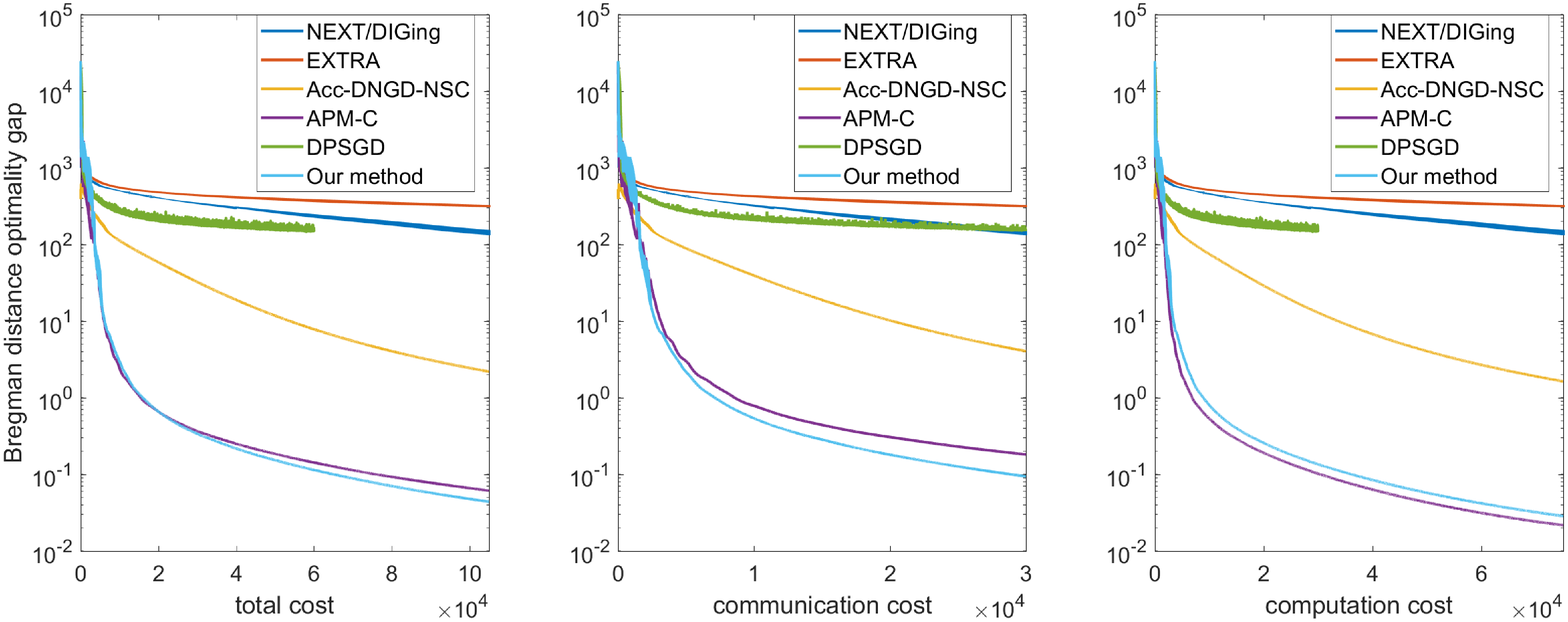}}
\centering{\includegraphics[scale=0.35]{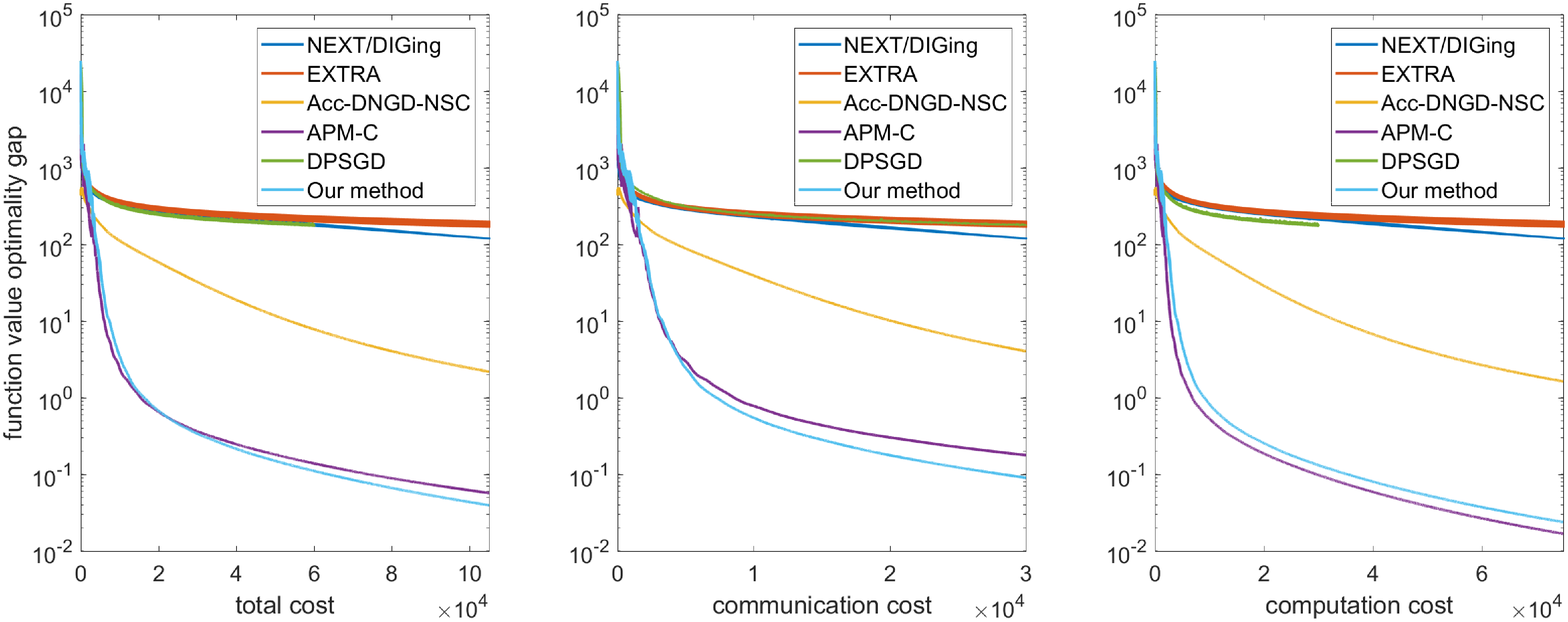}}

\vspace{-0.2cm}
\caption{Comparison for distributed algorithms for solving the decentralized logistic regression problem, in the setting where the communication time unit is 1 while the computation time unit for a full epoch of local data is 5.}
\label{fig:rescale_logreg}\vspace{-0.4cm}
\end{figure*}

\section{Additional Comments}
\label{sec_discussion_scaman}
\cite{scaman17optimal} presented optimal algorithms for decentralized optimization of strongly convex smooth functions. The functions considered in our paper are just convex (smooth). However, adding a small regularization (of the order of $\epsilon \norm{\x}^2/R^2$), the problem becomes strongly convex and the results of \citep{scaman17optimal}
apply. In so doing, one can show that the method in \citep{scaman17optimal} achieves an $\epsilon$-solution in $\Upbound{\left(1+\frac{1}{\sqrt{\eta}}\, \tau_c\right)\sqrt{\frac{L_f
R^2}{\epsilon}}\log\left(\frac{1}{\epsilon}\right)}$. However, this requires an accurate estimate of $R$ and the resulting rate differs from the lower bound for smooth convex functions by an extra log-factor ``$\log{1/\epsilon}$'', meaning that this is not optimal as our scheme. Also, more importantly, i) \cite{scaman17optimal}
requires the computation in closed form of the gradient of the Fenchel conjugate while our scheme does not have
this limitation; ii) the Lipschitz constant of the gradient of the Fenchel conjugate in the setting above will scale as
$1/\epsilon$ and thus the condition number of the dual problem is $O(1/\epsilon)$, which becomes arbitrarily large as $\epsilon$
decreases. As a consequence, \citep{scaman17optimal} significantly slow-downs in practice. This motivates our
design of distributed algorithms specifically for convex (but not strongly convex) functions.

\end{document}